\documentclass[a4paper,twoside,11pt，UTF8]{article}
\usepackage{amsmath,amssymb,amsfonts,amsthm,bm}

\usepackage{graphicx,xcolor}\graphicspath{{figures/}}
\usepackage{subfigure}
\usepackage{xcolor}
\usepackage{multirow}
\usepackage{longtable}
\usepackage{threeparttable}
\usepackage{booktabs}

\usepackage{chapterbib}
\usepackage{hyperref}
\usepackage{algorithm}
\usepackage[framemethod=tikz]{mdframed}  % shaded 环境
\usepackage{algorithmicx}
\usepackage{array}
\usepackage[ntheorem]{empheq}
\usepackage{algpseudocode} 

\usepackage{svg}

\usepackage{amsmath,latexsym,amssymb,amsfonts,amsbsy}
\usepackage{graphicx}
\usepackage{cite}
\usepackage{color}

\newtheorem{proposition}{Proposition}[section]
\newtheorem{theorem}{Theorem}[section]
\newtheorem{definition}{Definition}[section]
\newtheorem{lemma}{Lemma}[section]

\newtheorem{remark}{Remark}[section]

\title{Adaptive Randomized Extended Bregman-Kaczmarz Method for Combined Optimization Problems}

\author{ 
Zeyu Dong \\
Shanghai Research Institute for Intelligent Autonomous Systems, Tongji University,\\
Shanghai 200092, China.\\
Email: dongzeyu@tongji.edu.cn \\
	Aqin Xiao\\
	School of Mathematical Sciences, Tongji University, \\
	Shanghai 200092, China. \\
	Corresponding author, Email:xiaoaqin@tongji.edu.cn \\ 
    Guojian Yin\\
The Institute for Advanced Study, Shenzhen University, \\
Shenzhen, Guangdong 518001, China \\
Email:yin@szu.edu.cn\\
        and \\
	Junfeng Yin\\
	Key Laboratory of Intelligent Computing and Applications (Ministry of Education), \\
School of Mathematical Sciences, Tongji University, \\
Shanghai 200092, China \\
	Email:yinjf@tongji.edu.cn 
}

\begin{document}
\maketitle   

\begin{abstract}
Combined optimization problems that couple data-fidelity and regularization terms arise naturally in a wide range of inverse problems. In this paper, we study an adaptive randomized averaging block extended Bregman-Kaczmarz (aRABEBK) method for solving such problems. The proposed method incorporates iteration-wise relaxation parameters that are automatically adjusted using residual information, allowing for more aggressive step sizes without additional manual tuning. We establish a convergence theory for the proposed framework and derive expected linear convergence rate guarantees. Numerical experiments on both synthetic and real data sets for sparse and minimum-norm least-squares problems demonstrate that our aRABEBK method achieves faster convergence and improved robustness compared with state-of-the-art extended Kaczmarz and Bregman-Kaczmarz-type algorithms, including its nonadaptive counterpart.
\end{abstract}


\section{Introduction}
We consider the fundamental combined optimization problems as follows:
%We consider the fundamental combined optimization problems
\begin{equation}\label{eq:main}
 \min_{x \in \mathbb{R}^n} f(x), \quad \text{s.t.} \quad A x = \hat{y}, \quad \text{with} \quad \hat{y} =\arg\min_{y \in \mathcal{R}(A)} g^*(b - y),
\end{equation}    
%Here $A\in \mathbb{R}^{m\times n}$, $x\in \mathbb{R}^{n}$, $b\in\mathbb{R}^{m}$. 
where $A\in \mathbb{R}^{m\times n}$, $\hat{y}\in \mathbb{R}^{m}$, $b\in \mathbb{R}^m$, $\mathcal{R}(A)$ is the range of $A$, the objective function $f$ is assumed to be strongly convex but possibly nonsmooth, and  $g^*$ is a suitable data misfit function; see \cite{Schpfer2022ExtendedRK}. Problems of the form \eqref{eq:main} naturally arise when one seeks to integrate  prior structural information, such as sparsity or smoothness, into inverse problems and data-driven models.

A representative example is obtained when $$f(x) = \lambda \|x\|_1 + \frac{1}{2}\|x\|^2_2,$$ which is well known to promote sparse solutions for $\lambda>0$; see
\cite{chen2001atomic,donoho2006compressed,cai2009convergence}.  
If the data misfit adopts the standard least-squares form $$g^*(b-y) = \frac{1}{2}\|b - y\|_2^2,$$ then \eqref{eq:main} reduces to 
\begin{equation} \label{eq:main_partical}
    \min_{x \in \mathbb{R}^n} \lambda\|x\|_1 + \frac{1}{2}\|x\|_2^2, \quad \text{s.t.} \quad A x = \hat{y}, \quad \text{with} \quad \hat{y} = \arg \min_{y \in \mathcal{R}(A)} \frac{1}{2}\|b - y\|_2^2,
\end{equation}
which yields sparse least-squares reconstructions.

Another important special case arises when $f(x) = \tfrac{1}{2}\|x\|_2^2$, leading to the classical minimum-norm least-squares problem:
\begin{equation} \label{eq:main_minnorm}
    \min_{x \in \mathbb{R}^n} \tfrac{1}{2}\|x\|_2^2,
    \quad \text{s.t.} \quad A x = \hat{y},
    \quad \text{with} \quad 
    \hat{y} = \arg \min_{y \in \mathcal{R}(A)} \tfrac{1}{2}\|b - y\|_2^2,
\end{equation}
which plays a central role in ill-posed inverse problems and provides the minimum-norm solution consistent with the observed data.

Optimization models of the form \eqref{eq:main}--\eqref{eq:main_minnorm} arise in a broad range of modern applications, including compressed sensing \cite{cai2009linearized,yin2008Bregman}, signal and image processing \cite{dong2025surrogate,elad2010sparse,liang2020deep}, and machine learning and inverse problems \cite{adler2017solving,arridge2019solving,benning2018modern}. The efficient solution of these problems, particularly in large-scale or inconsistent settings, poses significant computational challenges. These challenges have, in turn, motivated the development of advanced iterative methods and projection-based algorithms capable of exploiting problem structure while maintaining stability and scalability.

In particular, when the linear system is consistent, i.e., $b \in \mathcal{R}(A)$, we have $\hat{y} = b$ in \eqref{eq:main_minnorm}, and the problem reduces to computing a minimum-norm solution of $Ax=b$. A classical method for this task is the Kaczmarz algorithm \cite{kaczmarz1937}, which iteratively projects onto the solution hyperplanes associated with individual rows of $A$. Owing to its simplicity and low per-iteration cost, the Kaczmarz method has inspired a wide range of randomized, block, and extended variants that significantly improve convergence behavior, especially for large-scale and possibly inconsistent systems.

\subsection{Background and related work}
In the standard Kaczmarz iteration, the rows of $A$ are visited cyclically.  At each iteration, the current iterate is projected onto the hyperplane $$H_{i_k}:\{x:  A^\top _ { i_k,: } x   = b _ { i_k }\}$$
defined by the active row index $i_k \in [m] = \{1,2,\ldots,m\}$.  Given an initial vector $x_0 \in \mathbb{R}^n$, the $(k+1)$-th iterate is computed as
\begin{equation} \label{eq:classicalKaczmarz}
        x _ { k + 1 }  = x _ { k } + \frac {b _ { i_k } - A _ { i_k,: } x _ { k } } { \| A _ { i_k,:  }\| _ { 2 } ^ { 2 } } \cdot (A _ { i_k,:  })^\top,
\end{equation}
which is the orthogonal projection of $x_k$ onto $H_{i_k}$.

To accelerate the convergence of the classical Kaczmarz method, Strohmer and Vershynin introduced in their seminal work \cite{strohmer2009randomized} the randomized Kaczmarz (RK) method, in which the active row is selected at random with probability proportional to its squared Euclidean norm. They established that the RK iteration enjoys an expected linear convergence rate, offering a substantial improvement over the cyclic deterministic Kaczmarz method.

Despite its favorable convergence properties, the RK method and its early variants generally fail to converge when the linear system is inconsistent. To address this limitation, Zouzias and Freris \cite{zouzias2013randomized} proposed the Randomized Extended Kaczmarz (REK) method, which augments the RK framework to restore convergence for inconsistent systems such as \eqref{eq:main_minnorm}. The key insight is that the right-hand side vector $b$ in \eqref{eq:main_minnorm} admits the orthogonal decomposition $$b = \hat{y}+z, \qquad \hat{y} \in \mathcal{R}(A) ,\quad  z\in  \mathcal{N}(A^\top),$$
which naturally yields two consistent subsystems:
$$Ax = \hat{y}  \qquad \text{and} \qquad A^{\top}z = 0.$$

The basic idea of the REK method is to perform two randomized Kaczmarz-type updates at each iteration, one associated with each of the consistent subsystems above. By jointly updating the primary variable $x$ and an auxiliary variable $z$, the inconsistency of the original system is effectively removed, and convergence to the minimum-norm least-squares solution is ensured. Analogous to the standard Kaczmarz iteration \eqref{eq:classicalKaczmarz}, the REK method updates the iterates according to
\begin{equation} \label{eq:extendedclassicalKaczmarz}
    \begin{split}
        z_{k+1}&=  z_{k}-\frac{(  A_{:,j_k})^\top   z_{k}}{\|A_{:,j_k}\|_2^2}  A_{:,j_k},\\
   x_{k+1}&= x_{k}-\frac{  A_{i_k,:}  x_{k}-  b_{i_k}+  z_{k+1,i_k}}{  \|A_{i_k,:}\|_2^2 }(  A_{i_k,:})^\top,
    \end{split}
\end{equation}
where the column index $j_k$ and the row index $i_k$ are independently selected according to the probability distributions
$$ \mathbb{P}(j_k = j) = \frac{\|A_{:,j}\|_2^2}{\|A\|^2_F}, \qquad \mathbb{P}(i_k = i) = \frac{\|A_{i,:}\|_2^2}{\|A\|^2_F}.$$

More recently, the  Bregman-Kaczmarz (BK) method \cite{tondji2023bregman,schopfer2019linear,lorenz2014sparse,tondji2023adaptive,tondji2023acceleration,Dong2025,xiao2025fast} have been developed as a generalization of the classical Kaczmarz framework. 
In this kind of approach, the Euclidean distance underlying the projection step is replaced by a Bregman distance induced by a strongly convex function $f$, thereby allowing the incorporation of structural priors and regularization effects.

The BK method is designed to solve the constrained optimization problem
\begin{equation} \label{eq:BK}
    \min_{x \in \mathbb{R}^n} f(x) \quad \mathrm{s.t.}\quad Ax=b,
\end{equation} 
which corresponds to the consistent case of problem~\eqref{eq:main}.  Let $\nabla f^{*}$ denote the gradient of the Fenchel conjugate $f^{*}$ of $f$, which maps the dual iterate to the corresponding primal iterate. Starting from the initialization $x_0^{*}=0$ and $x_0 = \nabla f^{*}(x_0^{*})$, the $(k+1)$-th iterate is computed via
\begin{equation} \label{eq:Bregman-Kaczmarz2}
    \begin{split}
        &x _ { k + 1 } ^ { * } = x _ { k } ^ { * } +  \frac {b _ { i_k } - A _ { i_k,: } x _ { k } } { \| A _ { i_k,:  }\| _ { 2 } ^ { 2 } } \cdot (A _ { i_k,:  })^\top, \\
        &x _ { k + 1 } =\nabla f^{*}(x^{*}_{k+1}).
    \end{split}
\end{equation}

The update scheme \eqref{eq:Bregman-Kaczmarz2} consists of two stages.  
The first stage, referred to as the \emph{dual-space update} \cite{tondji2023bregman}, resembles the standard Kaczmarz projection \eqref{eq:classicalKaczmarz} but is performed in the dual variable $x_k^{*}$.  The second stage, known as  the \emph{primal-space update} \cite{tondji2023bregman}, maps the updated dual iterate $x_{k+1}^{*}$ back to the primal space via the gradient of the Fenchel conjugate, $\nabla f^{*}$, yielding the next primal iterate $x_{k+1}$.

When the generating function is chosen as $f(x)=\tfrac{1}{2}\|x\|_2^2$, its conjugate satisfies $f^*=f$ and $\nabla f^*(x)=x$. In this setting, the BK iteration reduces exactly to the standard Kaczmarz method \eqref{eq:classicalKaczmarz}, and the iterates converge to the minimum-norm solution of the consistent linear system.

More generally, different choices of the function $f$ lead to Kaczmarz-type algorithms with distinct structural properties. For instance, selecting $$ f(x)=\lambda\|x\|_1+\tfrac{1}{2}\|x\|_2^2 $$
in \eqref{eq:BK} promotes sparsity in the solution. In this case, the gradient of the conjugate $\nabla f^{*}$ coincides with the soft-thresholding operator $S_{\lambda}(\cdot)$, and the primal update takes the form $$ x_{k+1}=S_{\lambda}(x_{k+1}^{*}). $$
Consequently, the BK method specializes to the sparse Kaczmarz algorithm \cite{lorenz2014sparse}.

To extend the Bregman-Kaczmarz method to solve inconsistent linear systems, Schöpfer et al.~\cite{Schpfer2022ExtendedRK} proposed the \emph{randomized extended Bregman-Kaczmarz} (REBK) method for solving \eqref{eq:main}. The REBK algorithm can be viewed as a synthesis of the randomized extended Kaczmarz (REK) strategy \eqref{eq:extendedclassicalKaczmarz} and the BK framework \eqref{eq:Bregman-Kaczmarz2}, thereby combining inconsistency handling with structure-promoting regularization induced by Bregman distances. Its iterative scheme is given by
\begin{equation} \label{eq:REBK_scheme}
    % \begin{aligned}
    \begin{split}
   z_{k+1}^*&=  z_{k}^*-\frac{(  A_{:,j_k})^\top   z_{k}}{\|A_{:,j_k}\|_2^2}  A_{:,j_k},\\
   z_{k+1}&= \nabla g^*(z_{k+1}^*);\\
   x_{k+1}^*&= x_{k}^*-\frac{  A_{i_k,:}  x_{k}-  b_{i_k}+  z^*_{k+1,i_k}}{  \|A_{i_k,:}\|_2^2 }(  A_{i_k,:})^\top, \\
     x_{k+1}&= \nabla f^*(x_{k+1}^*).
    \end{split}
    % \end{aligned}
\end{equation}

In parallel, averaging block techniques have emerged as an effective mechanism for accelerating Kaczmarz-type methods by exploiting multiple projections per iteration. 
Necoara~\cite{necoara2019faster} introduced the \emph{randomized averaging block Kaczmarz} (RABK) algorithm, in which a block of rows is selected at each iteration, multiple randomized Kaczmarz projections are performed onto the associated hyperplanes, and a convex combination of these projections is used to define a new search direction with an appropriate stepsize.

Building upon this idea, Du et al. \cite{du2020randomized} extended the averaging block strategy to the extended Kaczmarz framework and proposed the \emph{randomized extended averaging block Kaczmarz} (REABK) method for solving the minimum-norm least-squares problem \eqref{eq:main_minnorm}. 
More recently, in our work \cite{dong2025relaxed}, we further unified the averaging block mechanism with the extended Bregman-Kaczmarz framework, leading to the \emph{randomized averaging block extended Bregman-Kaczmarz} (RABEBK) method, together with its constant-relaxation variant (cRABEBK), for solving the general combined optimization problem \eqref{eq:main}.

\subsection{Main contributions} In this work, we take a further step by incorporating \emph{adaptive relaxation parameters} into the RABEBK framework, yielding the \emph{adaptive randomized averaging block extended Bregman-Kaczmarz} (aRABEBK) methods. The relaxation parameters are updated at each iteration using an extrapolation-inspired strategy \cite{necoara2019faster}, allowing the algorithm to make more aggressive yet stable progress along the block-averaged update directions, which significantly enhances convergence performance.

The main contributions of this paper can be summarized as follows:
\begin{enumerate}
  \item We propose a class of \emph{adaptive randomized averaging block extended Bregman-Kaczmarz} (aRABEBK) methods for solving the combined optimization problem \eqref{eq:main}. 
  The algorithms incorporate iteration-dependent relaxation parameters that are automatically adjusted based on residual information, effectively exploiting the block structure of the updates.
  \item We establish a comprehensive convergence theory for the proposed aRABEBK methods. 
  We prove expected linear convergence rates and derive explicit error bounds that capture the effect of the adaptive relaxation strategy.
   \item We introduce a general adaptive relaxation framework that can be integrated with several existing Kaczmarz-type schemes. In particular, for the combined optimization problem \eqref{eq:main}, aRABEBK can be interpreted as an adaptive extension of our previous randomized averaging block extended Bregman-Kaczmarz (RABEBK) methods \cite{dong2025relaxed}, whereas for the minimum-norm least-squares problem \eqref{eq:main_minnorm}, the resulting algorithm can be viewed as an adaptive counterpart of the randomized extended average block Kaczmarz (REABK) method~\cite{du2020randomized}.
\end{enumerate}

\subsection{Organization}
The remainder of this paper is organized as follows. 
In Section~\ref{sec:prelim}, we introduce the notation and preliminaries used throughout the paper. 
Section~\ref{sec:RABEBK-a} provides a brief review of the randomized averaging block extended Bregman-Kaczmarz (RABEBK) method, including its constant-relaxation variant (cRABEBK), and then presents the adaptive extension, aRABEBK. Section~\ref{sec:RABEBK-Convergence} is devoted to a detailed convergence analysis of aRABEBK. Numerical experiments are reported in Section~\ref{sec:numerical experiments} to demonstrate the efficiency of the proposed method. 
Finally, Section~\ref{sec:conclusion} concludes the paper and discusses potential directions for future work.

\section{ Preliminaries}\label{sec:prelim}
In this section, we introduce notations and recall fundamental concepts from convex analysis that will be used throughout the paper.

For a vector $x \in \mathbb{R}^n$, we denote by $x^\top$ its transpose, and by $\|x\|_1$ and $\|x\|_2$ its $\ell_1$- and $\ell_2$-norms, respectively.  For a matrix $A \in \mathbb{R}^{m \times n}$, we denote by $A_{i,:}$ and $A_{:,j}$ the $i$-th row and $j$-th column, respectively; $A^\top$ the transpose; $A^\dagger$ the Moore-Penrose general inverse; $\mathcal{R}(A)$ the range space; $\|A\|_2$ the spectral norm; $\|A\|_F$ the Frobenius norm; and $\sigma_{\max}(A)$ and $\sigma_{\min}(A)$ the largest and smallest nonzero singular values of $A$, respectively.

For index sets $\mathcal{I} \subseteq [m]$ and $\mathcal{J} \subseteq [n]$, we define $A_{\mathcal{I},:}$, $A_{:,\mathcal{J}}$, and $A_{\mathcal{I},\mathcal{J}}$ as the row submatrix indexed by $\mathcal{I}$, the column submatrix indexed by $\mathcal{J}$, and the submatrix consisting of rows indexed by $\mathcal{I}$ and columns indexed by $\mathcal{J}$, respectively.

We call $\{\mathcal{I}_1,\mathcal{I}_2,\ldots,\mathcal{I}_s\}$ a partition of $[m]$ if
$\mathcal{I}_i \cap \mathcal{I}_j = \emptyset$ for $i \neq j$ and $\bigcup_{i=1}^s \mathcal{I}_i = [m]$.
Similarly, $\{\mathcal{J}_1,\mathcal{J}_2,\ldots,\mathcal{J}_t\}$ is a partition of $[n]$ if
$\mathcal{J}_i \cap \mathcal{J}_j = \emptyset$ for $i \neq j$ and $\bigcup_{j=1}^t \mathcal{J}_j = [n]$.
The cardinality of a set $\mathcal{I} \subseteq [m]$  is denoted by $|\mathcal{I}|$.

The soft shrinkage operator $S_\lambda(\cdot)$ is defined componentwise as $$(S_{\lambda}(x))_{j}=\max \left\{ |x_{j}|- \lambda ,0 \right\} \cdot \mathrm{sign}(x_{j}),$$ where $\mathrm{sign}(\cdot)$ denotes the sign function.

Let $f: \mathbb{R}^n\rightarrow \mathbb{R}$ be a convex function. The \emph{subdifferential} of $f$ at any point $x \in \mathbb{R}^n$ is defined by
\begin{equation*}
    \partial f ( x ) \stackrel{\text { def }}{=} \{ x ^ { * }  \in \mathbb{R} ^ { n } | f ( y ) \geq f ( x ) + \langle  x ^ { * } , y - x \rangle  , \forall y\in \mathbb{R}^n\}.
\end{equation*}

The \emph{Fenchel conjugate} of $f$, denoted $f^*:\mathbb{R}^n \rightarrow \mathbb{R}$, is defined as
$$f^{*}(x^{*})\stackrel{\text { def }}{=}\sup_{y \in \mathbb{R}^{n}} \{\langle x^{*},y \rangle -f(y)\}.$$

If $f$ is differentiable at $x$, we denote its gradient by $\nabla f(x)$, and in this case the subdifferential reduces to a singleton: $\partial f(x) = \{ \nabla f(x)\}$\cite{schopfer2019linear}.

A convex function $f:\mathbb{R}^n \rightarrow \mathbb{R}$ is said to be \emph{$\alpha$-strongly convex} if there exists a constant $\alpha > 0$ such that for all $x, y \in \mathbb{R}^n$ and any $x^* \in \partial f(x)$,
\begin{equation*}
    f(y)\geq f(x)+ \langle x^{*},y-x \rangle + \frac{\alpha}{2}\cdot \|y-x\|_{2}^{2}.
\end{equation*}

\begin{lemma}[\cite{rockafellar1998variational}]
If $f:\mathbb{R}^n \rightarrow \mathbb{R}$ is $\alpha$-strongly convex, then its conjugate function $f^*$ is differentiable with a $1/\alpha$-Lipschitz-continuous gradient, i.e.
\begin{equation*}
    \| \nabla f^{*}(x^{*})- \nabla f^{*}(y^{*})\|_{2}\leq \frac{1}{\alpha}\cdot \|x^{*}-y^{*}\|_{2}, \quad \forall x^*, y^* \in \mathbb{R}^n,
\end{equation*}
which implies the estimate
\begin{equation} \label{eq:alpha-strongly convex}
   f^{*}(y^{*})\leq f^{*}(x^{*})+\langle \nabla f^{*}(x^{*}),y^{*}-x^{*}\rangle + \frac{1}{2\alpha}\|y^{*}-x^{*}\|_{2}^{2}. 
\end{equation}
\end{lemma}
\begin{definition}[Bregman distance \cite{bregman1967relaxation}]
Let $f$ be a strongly convex function and $x^* \in \partial f(x)$ for some $x \in \mathbb{R}^n$. The \emph{Bregman distance} between $x$ and $y \in \mathbb{R}^n$ with respect to $f$ and $x^*$ is defined as
\begin{equation} \label{eq:Bregman distance}
    D_{f}^{x^{*}}(x,y)\stackrel{\text { def }}{=}{f(y)-f(x)- \langle x^{*},y-x \rangle } =f^{*}(x^{*})- \langle x^{*},y \rangle +f(y).
\end{equation}
\end{definition}

\begin{lemma}[\cite{schopfer2019linear}] \label{lem:Euclid-Bregman}
Let $f:\mathbb{R}^n \rightarrow \mathbb{R}$ be $\alpha$-strongly convex. Then, for all
$x,y \in \mathbb{R}^n$ and any subgradient $x^* \in \partial f(x)$, the associated
Bregman distance satisfies
$$ D _ { f } ^ { x^* } ( x , y ) \geq \frac { \alpha } { 2 } \| x - y \| _ { 2 } ^ { 2 }.$$
Moreover, $D _ { f } ^ { x^* } ( x , y )=0$ if and only if $x=y$.
\end{lemma}

\section{The Proposed Methods} \label{sec:RABEBK-a}
In our recent work \cite{dong2025relaxed}, we introduced the randomized averaging block extended Bregman-Kaczmarz (RABEBK) method and its constant-relaxation variant (cRABEBK) for solving the combined optimization problem \eqref{eq:main}. While these methods already exhibit improved convergence behaviour by exploiting block-averaged projections, their performance is still influenced by the choice of relaxation parameters.

The objective of the present work is to further enhance the efficiency of the RABEBK framework by incorporating iteration-dependent adaptive relaxation parameters. Motivated by extrapolation-based acceleration strategies \cite{necoara2019faster}, we design an adaptive scheme that dynamically adjusts the relaxation parameters according to the iteration-wise residual information, enabling more aggressive yet stable progress at each iteration.

Before introducing the proposed adaptive randomized averaging block extended Bregman-Kaczmarz method, we first briefly review the RABEBK method and its constant-relaxation variant cRABEBK for completeness.

\subsection{Review of the RABEBK Method and Its Constant-Relaxation Variant}
\label{sec:RABEBK-d}
The randomized averaging block extended Bregman-Kaczmarz (RABEBK) method and its constant-relaxation variant (cRABEBK) form the foundation for the adaptive scheme developed in this work. They extend the classical Kaczmarz framework to a broad class of combined optimization problems by integrating randomized averaging block updates with extended Bregman iterations.

The key idea of the RABEBK method lies in embedding the extended Bregman-Kaczmarz iteration into a randomized block-averaging strategy. At each iteration, multiple row and column blocks are selected at random, and the corresponding extended Bregman updates are computed and aggregated to form a single update direction. This construction allows the method to exploit block-wise randomization while incorporating the geometry induced by convex regularization via Bregman distances, thereby yielding an efficient, scalable, and structure-preserving iterative algorithm.

In what follows, we formulate the RABEBK method for solving the combined optimization problem \eqref{eq:main}.
We begin by considering the auxiliary problem
\begin{equation} \label{eq:SLTW17}
    \min_{y \in \mathcal{R}(A)} g^*(b - y).
\end{equation}
Under the assumption that $g$ is strongly convex, it is straightforward to verify that the dual problem of \eqref{eq:SLTW17} takes the form
\begin{equation} \label{eq:BK01}
    \min_{z \in \mathbb{R}^m} h(z): = g(z)-b^\top z \quad \mathrm{s.t.}\quad A^\top z=0.
\end{equation}
Moreover, the strong convexity of $g$ guarantees zero duality gap between the primal-dual pair~\eqref{eq:SLTW17}-\eqref{eq:BK01}.

Applying the randomized averaging block variant of the BK scheme
\eqref{eq:Bregman-Kaczmarz2} to the constrained dual problem~\eqref{eq:BK01}
yields the following iterative updates:
\begin{equation} \label{eq:BK03}
    \begin{split}
        &\tilde{z}_{k+1}^* = \tilde{z}_k^* -\frac{1} {\lVert A_{:,\mathcal{J}_{j_k}}\rVert^2_F} A_{:,\mathcal{J}_{j_k}}(A_{:,\mathcal{J}_{j_k}})^\top z_{k}, \\
        &z_{k+1} = \nabla h^*(\tilde{z}_{k+1}^*), 
    \end{split}
 \end{equation}
with initialization \( \tilde{z}_0^* = 0 \). At each iteration $k$, the column block $\mathcal{J}_{j_k}$ is selected independently according to the probability distribution $\mathbb{P}(\mathcal{J}_{j_k})=\|A_{:,\mathcal{J}_{j_k}}\|_F^2/\|A\|_F^2$.

Noting that $$ \nabla h^*(\tilde{z}_k^*) = \nabla g^*(\tilde{z}_k^* + b),$$ we introduce the shifted dual variable \( z_k^* := \tilde{z}_k^* + b \), which allows us to rewrite \eqref{eq:BK03} equivalently as
\begin{equation} \label{eq:BK04}
    \begin{split}
        &z_{k+1}^* = z_{k}^*-\frac{1} {\lVert A_{:,\mathcal{J}_{j_k}}\rVert^2_F} A_{:,\mathcal{J}_{j_k}}(A_{:,\mathcal{J}_{j_k}})^\top z_{k}, \\
        &z_{k+1} = \nabla g^*(z_{k+1}^*), 
    \end{split}
 \end{equation}
with initialization \( z_0^* = b \).

On the other hand, given $\hat{y} = b-z^*$, the original combined optimization problem
\eqref{eq:main} requires solving
\begin{equation} \label{eq:BK02}
    \min_{ x\in \mathbb{R}^n} f(x) \quad \mathrm{s.t.}\quad Ax=\hat{y},
\end{equation}
which can be addressed by using the randomized averaging block Bregman-Kaczmarz
scheme in a manner analogous to \eqref{eq:BK03}.

Consequently, solving the combined optimization problem \eqref{eq:main} can be interpreted as an alternating solution process that applies randomized averaging block Bregman-Kaczmarz updates to two constrained subproblems, namely~\eqref{eq:BK01} and~\eqref{eq:BK02}, which are coupled through the relation $b = \hat{y} + z^*$. From this alternating primal-dual perspective, a natural algorithmic realization leads to the \emph{randomized averaging block extended Bregman-Kaczmarz} (RABEBK) method.

We now arrive at the procedure of the RABEBK method. Given initialization $z_0^{*}=b$, $z_0=\nabla g^*(z_0^*)$, $x_0^{*}=0$ and $x_0=\nabla f^{*}(x_0^*)$. At the $k$-th iteration, RABEBK performs the updates:
\begin{equation} \label{eq:RABEBK_scheme}
    \begin{aligned}
   z_{k+1}^*&=  z_{k}^*-\sum_{l\in\mathcal{J}_{j_k}}\omega_{l}\frac{(  A_{:,l})^\top   z_{k}}{\|A_{:,l}\|_2^2}  A_{:,l},\quad \omega_l=\frac{ \|A_{:,l}\|_2^2}{\|  A_{:,\mathcal{J}_{j_k}}\|_F^2},\\
   z_{k+1}&= \nabla g^*(z_{k+1}^*);\\
   x_{k+1}^*&= x_{k}^*-\sum_{q\in\mathcal{I}_{i_k}}\omega_q\frac{  A_{q,:}  x_k-  b_q+  z^*_{k+1,q}}{  \|A_{q,:}\|_2^2 }(  A_{q,:})^\top, \quad \omega_q=\frac{  \|A_{q,:}\|_2^2}{\|  A_{\mathcal{I}_{i_k},:}\|_F^2},\\
     x_{k+1}&= \nabla f^*(x_{k+1}^*).
    \end{aligned}
\end{equation}
Here, the column block  \(\mathcal{J}_{j_k}\) and the row block \(\mathcal{I}_{i_k}\) are selected at random, with probabilities proportional to the Frobenius norms of the associated submatrices, namely,
$$\mathbb{P}(\mathcal{J}_{j_k})=\frac{\|A_{:,\mathcal{J}_{j_k}}\|_F^2}{\|A\|_F^2},
\qquad
\mathbb{P}(\mathcal{I}_{i_k})=\frac{\|A_{\mathcal{I}_{i_k},:}\|_F^2}{\|A\|_F^2}.
$$
which assign higher sampling probability to blocks with larger Frobenius norms.

To further enhance the numerical stability and convergence behavior of RABEBK, two constant relaxation parameters, corresponding respectively to the dual and primal updates, were introduced in \cite{dong2025relaxed}, leading to the \emph{constant-relaxation randomized averaging block extended Bregman-Kaczmarz} (cRABEBK) method.  The resulting iterative scheme is given by
\begin{equation}\label{eq:rRABEBK_scheme}
    \begin{aligned}
   z_{k+1}^*&=  z_{k}^*-\alpha^{(z)}\left(\sum_{l\in\mathcal{J}_{j_k}}\omega_{l}\frac{(  A_{:,l})^\top   z_{k}}{\|A_{:,l}\|_2^2}  A_{:,l}\right),\quad \omega_l=\frac{ \|A_{:,l}\|_2^2}{\|  A_{:,\mathcal{J}_{j_k}}\|_F^2},\\
   z_{k+1}&= \nabla g^*(z_{k+1}^*);\\
   x_{k+1}^*&= x_{k}^*-\alpha^{(x)}\left(\sum_{q\in\mathcal{I}_{i_k}}\omega_q\frac{  A_{q,:}  x_k-  b_q+  z^*_{k+1,q}}{  \|A_{q,:}\|_2^2 }(  A_{q,:})^\top \right), \quad \omega_q=\frac{  \|A_{q,:}\|_2^2}{\|  A_{\mathcal{I}_{i_k},:}\|_F^2},\\
     x_{k+1}&= \nabla f^*(x_{k+1}^*).
    \end{aligned}
\end{equation}
In \cite{dong2025relaxed}, the constant relaxation parameters associated with the dual and primal updates are chosen, respectively, as
$$ \alpha^{(z)} = \frac{\mu_g}{\beta^{\mathcal{J}}_{\max}}\quad\text{and}\quad   \alpha^{(x)} = \frac{\mu_f}{\beta^{\mathcal{I}}_{\max}}, $$ 
where  \begin{equation}\label{eq: beta_def}
\beta^{\mathcal{I}}_{\max} := \max_{i_k \in [s]} \frac{\sigma^2_{\max} \left( A_{\mathcal{I}_{i_k},:} \right)}{\| A_{\mathcal{I}_{i_k},:} \|^2_F} \quad\text{and}\quad \beta^{\mathcal{J}}_{\max} := \max_{j_k \in [t]} \frac{\sigma^2_{\max} \left( A_{:,\mathcal{J}_{j_k}} \right)}{\| A_{:,\mathcal{J}_{j_k}} \|^2_F}.
\end{equation}
Here, $\mu_f$ and $\mu_g$ denote the strong convexity constants of $f$ and $g$, respectively, and $\sigma_{\max}(\cdot)$ denotes the largest singular value.

Although the constant-relaxation strategy leads to a simple and easily implementable scheme, its convergence behaviour may vary significantly across different problem instances, since the relaxation parameter remains fixed throughout the iterations. To address this limitation, we develop an adaptive relaxation strategy in which the parameter is dynamically updated based on the current residual information. The design and analysis of this adaptive approach constitute the main focus of the next subsection.

\subsection{Adaptive RABEBK method} \label{sec:RABEBK-ada}
While the introduction of constant relaxation parameters significantly improves the robustness and convergence behaviour of RABEBK in practice, the performance of the cRABEBK method remains highly sensitive to the choice of these parameters. In particular, the theoretically admissible values of $\alpha^{(z)}$ and $\alpha^{(x)}$ depend on global spectral quantities such as $\beta^{\mathcal{J}}_{\max}$ and $\beta^{\mathcal{I}}_{\max}$, which are often difficult to estimate accurately for large-scale or highly structured problems. Moreover, relying on conservative global bounds may lead to overly cautious step sizes, thereby limiting the practical efficiency of the algorithm.

Motivated by these observations and inspired by recent advances in extrapolation-based
acceleration techniques for block Kaczmarz-type methods \cite{necoara2019faster},
we propose to replace the constant relaxation parameters in cRABEBK with
iteration-dependent adaptive relaxation parameters. The resulting method, termed the
\emph{adaptive randomized averaging block extended Bregman-Kaczmarz} (aRABEBK) method,
automatically adjusts the relaxation parameters using locally available residual
information. This mechanism allows for more aggressive updates whenever permitted by the local
geometry, while maintaining numerical stability. As a result, the need for a priori parameter tuning is alleviated, and substantially accelerated convergence is observed in practice.

Specifically, the constant-relaxed iteration~\eqref{eq:rRABEBK_scheme} can be enhanced by replacing the constant relaxation parameters $\alpha^{(z)}$ and $\alpha^{(x)}$ with two adaptive relaxation parameters, leading to the following iteration scheme:
\begin{equation}\label{eq:aRABEBK_scheme}
    \begin{aligned}
   z_{k+1}^*&=  z_{k}^*-\alpha_k^{(z)}\left(\sum_{l\in\mathcal{J}_{j_k}}\omega_{l}\frac{(  A_{:,l})^\top   z_{k}}{\|A_{:,l}\|_2^2}  A_{:,l}\right),\quad \omega_l=\frac{ \|A_{:,l}\|_2^2}{\|  A_{:,\mathcal{J}_{j_k}}\|_F^2},\\
   z_{k+1}&= \nabla g^*(z_{k+1}^*);\\
   x_{k+1}^*&= x_{k}^*-\alpha_k^{(x)}\left(\sum_{q\in\mathcal{I}_{i_k}}\omega_q\frac{  A_{q,:}  x_k-  b_q+  z^*_{k+1,q}}{  \|A_{q,:}\|_2^2 }(  A_{q,:})^\top \right), \quad \omega_q=\frac{  \|A_{q,:}\|_2^2}{\|  A_{\mathcal{I}_{i_k},:}\|_F^2},\\
     x_{k+1}&= \nabla f^*(x_{k+1}^*).
    \end{aligned}
\end{equation}
Here, \(\alpha_k^{(z)}\) and \(\alpha_k^{(x)}\) are iteration-dependent relaxation
parameters that are updated adaptively as the algorithm proceeds.

Now, a key question concerns how to choose the relaxation parameters in a principled
and effective manner. As shown earlier, solving the combined optimization problem \eqref{eq:main} can be decomposed into two constrained subproblems, namely \eqref{eq:BK01} and
\eqref{eq:BK02}, associated with the auxiliary variable $z$ and the primary variable
$x$, respectively. It therefore naturally motivates the introduction of two distinct adaptive relaxation
parameters: $\alpha_k^{(z)}$ for the update of $z_{k+1}$ corresponding to \eqref{eq:BK01}, and $\alpha_k^{(x)}$ for the update of $x_{k+1}$ corresponding to \eqref{eq:BK02}.

To this end, we adopt an extrapolation-based stepsize strategy inspired by
\cite{necoara2019faster}, originally developed for consistent linear systems, and extend
it here to the present block-structured and regularized setting. The key idea is to choose the relaxation parameters by exploiting locally available residual information, with the goal of adaptively balancing the magnitude of the blockwise update and improving the robustness of the iteration.

Specifically, the adaptive relaxation parameters in our aRABEBK method are defined as
  \begin{equation} \label{eq:adaptive_z}
   \alpha_k^{(z)}=   \delta_z \frac{ \lVert A_{:,\mathcal{J}_{j_k}}\rVert^2_F\|r_k^{(z)}\|_2^2}{ \|A_{:,\mathcal{J}_{j_k}}r_k^{(z)}\|_2^2}, \qquad r_k^{(z)}=(A_{:,\mathcal{J}_{j_k}})^\top z_{k},
\end{equation}
and
\begin{equation}\label{eq:adaptive_x}
\alpha_k^{(x)}= \delta_x \frac{ \lVert( A_{\mathcal{I}_{i_k},:})^\top\rVert^2_F \|r_k^{(x)}\|_2^2}{\| (A_{\mathcal{I}_{i_k},:})^\top r_k^{(x)}\|^2_2}, \qquad r_k^{(x)}=b_{\mathcal{I}_{i_k}}-A_{\mathcal{I}_{i_k},:}x_{k} -z_{k+1,\mathcal{I}_{i_k}}^*,
\end{equation}
where $\delta_z>0$ and $\delta_x>0$ are user-specified scaling parameters that control
the overall aggressiveness of the extrapolation.

Moreover, the choices in \eqref{eq:adaptive_z}--\eqref{eq:adaptive_x} admit an interpretation in terms of exact line search:
when $\delta_z=\delta_x=1$, they provide an upper bound on the corresponding exact line search
stepsizes for the local quadratic residual models, and hence lead to more aggressive
updates. We formalize this connection in the following proposition \ref{prop:exact_line_search_aRABEBK}.

\begin{proposition}[Exact line search in block form] \label{prop:exact_line_search_aRABEBK}
Let $\mathcal{J}_{j_k}$ and $\mathcal{I}_{i_k}$ be the randomly selected column and row blocks at iteration $k$.
Define the auxiliary and primary residuals
\[
r_k^{(z)} := (A_{:,\mathcal{J}_{j_k}})^\top z_k,
\qquad
r_k^{(x)} := b_{\mathcal{I}_{i_k}} - A_{\mathcal{I}_{i_k},:}x_k - z^{*}_{k+1,\mathcal{I}_{i_k}} .
\]
Consider the one-dimensional line search for the local quadratic residual model
\[
\min_{\alpha>0}\ \phi_z(z_k-\alpha d_k^{(z)})
=\frac12\bigl\|(A_{:,\mathcal{J}_{j_k}})^\top(z_k-\alpha d_k^{(z)})\bigr\|_2^2,
\]
where $d_k^{(z)}$ is the corresponding block-averaged update direction. Then the minimizer is attained at
\begin{equation} \label{eq:exact_line_search_19}
\alpha_{\mathrm{ex}}^{(z)}
=
\frac{\bigl\|A_{:,\mathcal{J}_{j_k}}\,r_k^{(z)}\bigr\|_2^{2}\,\bigl\|A_{:,\mathcal{J}_{j_k}}\bigr\|_F^{2}}
{\bigl\|(A_{:,\mathcal{J}_{j_k}})^\top A_{:,\mathcal{J}_{j_k}}\,r_k^{(z)}\bigr\|_2^{2}}.
\end{equation}
Analogously, the exact line search stepsize for the primary residual model is given by
\begin{equation}\label{eq:exact_line_search_20}
\alpha_{\mathrm{ex}}^{(x)}
= 
\frac{\bigl\|(A_{\mathcal{I}_{i_k},:})^\top r_k^{(x)}\bigr\|_2^{2}\,\bigl\|A_{\mathcal{I}_{i_k},:}\bigr\|_F^{2}}
{\bigl\|A_{\mathcal{I}_{i_k},:}(A_{\mathcal{I}_{i_k},:})^\top r_k^{(x)}\bigr\|_2^{2}}.
\end{equation}
\end{proposition}

\begin{proof}
We first derive \eqref{eq:exact_line_search_19} for $\alpha_{\mathrm{ex}}^{(z)}$; the derivation of \eqref{eq:exact_line_search_20}
follows by an analogous argument.

For the selected column block $\mathcal{J}_{j_k}$, consider the quadratic residual function
\[
\phi_z(z) := \frac12\bigl\|(A_{:,\mathcal{J}_{j_k}})^\top z\bigr\|_2^2,
\]
which measures the violation of the block consistency condition $(A_{:,\mathcal{J}_{j_k}})^\top z=0$.

In view of the update scheme \eqref{eq:aRABEBK_scheme}, the block-averaged auxiliary update can be written as
\[
z_{k+1}^* = z_k^* - \alpha\, d_k^{(z)},
\qquad
d_k^{(z)}=\sum_{l\in\mathcal{J}_{j_k}} \omega_l
\frac{(A_{:,l})^\top z_k}{\|A_{:,l}\|_2^2}\,A_{:,l}
= \frac{A_{:,\mathcal{J}_{j_k}}(A_{:,\mathcal{J}_{j_k}})^\top z_k}{\|A_{:,\mathcal{J}_{j_k}}\|_F^2},
\]
where $\omega_l=\|A_{:,l}\|_2^2/\|A_{:,\mathcal{J}_{j_k}}\|_F^2$.

Minimizing $\phi_z$ along the direction $d_k^{(z)}$ yields
\[
\min_{\alpha>0}\ \phi_z(z_k-\alpha d_k^{(z)})
=\frac12\bigl\|(A_{:,\mathcal{J}_{j_k}})^\top(z_k-\alpha d_k^{(z)})\bigr\|_2^2.
\]
Expanding the right-hand side gives
\[
\phi_z(z_k-\alpha d_k^{(z)})
=\frac12\|r_k^{(z)}\|_2^2
-\alpha\bigl\langle r_k^{(z)},(A_{:,\mathcal{J}_{j_k}})^\top d_k^{(z)}\bigr\rangle
+\frac{\alpha^2}{2}\bigl\|(A_{:,\mathcal{J}_{j_k}})^\top d_k^{(z)}\bigr\|_2^2.
\]
Setting the derivative with respect to $\alpha$ to zero yields
\[
\alpha_{\mathrm{ex}}^{(z)}
=\frac{\bigl\langle r_k^{(z)},(A_{:,\mathcal{J}_{j_k}})^\top d_k^{(z)}\bigr\rangle}
{\bigl\|(A_{:,\mathcal{J}_{j_k}})^\top d_k^{(z)}\bigr\|_2^2}.
\]
Using the explicit form of $d_k^{(z)}$, we have
\[
(A_{:,\mathcal{J}_{j_k}})^\top d_k^{(z)}
=
\frac{(A_{:,\mathcal{J}_{j_k}})^\top A_{:,\mathcal{J}_{j_k}}(A_{:,\mathcal{J}_{j_k}})^\top z_k}
{\|A_{:,\mathcal{J}_{j_k}}\|_F^2}.
\]
Substituting the above expression and $r_k^{(z)}=(A_{:,\mathcal{J}_{j_k}})^\top z_k$ into the formula for $\alpha_{\mathrm{ex}}^{(z)}$
yields \eqref{eq:exact_line_search_19}. The expression \eqref{eq:exact_line_search_20} follows by applying the same argument to the primary
residual function
\[
\phi_x(x)=\frac12\bigl\|b_{\mathcal{I}_{i_k}}-A_{\mathcal{I}_{i_k},:}x-z_{k+1,\mathcal{I}_{i_k}}^*\bigr\|_2^2,
\]
along the corresponding block-averaged update direction induced by $A_{\mathcal{I}_{i_k},:}$.
\end{proof}

We compare the adaptive stepsize \eqref{eq:adaptive_z} with the exact line search stepsize \eqref{eq:exact_line_search_19}.
When $\delta_z=1$, they read
\[
\alpha_k^{(z)}=\frac{\|A_{:,\mathcal{J}_{j_k}}\|_F^{2}\,\|r_k^{(z)}\|_2^{2}}
{\|A_{:,\mathcal{J}_{j_k}}\,r_k^{(z)}\|_2^{2}},
\qquad
\alpha_{\mathrm{ex}}^{(z)}=\frac{\|A_{:,\mathcal{J}_{j_k}}\,r_k^{(z)}\|_2^{2}\,\|A_{:,\mathcal{J}_{j_k}}\|_F^{2}}
{\|(A_{:,\mathcal{J}_{j_k}})^\top A_{:,\mathcal{J}_{j_k}}\,r_k^{(z)}\|_2^{2}}.
\]
By the Cauchy--Schwarz inequality,
\[
\bigl\langle r_k^{(z)},(A_{:,\mathcal{J}_{j_k}})^\top A_{:,\mathcal{J}_{j_k}}\,r_k^{(z)}\bigr\rangle^{2}
\le \|r_k^{(z)}\|_2^{2}\,\|(A_{:,\mathcal{J}_{j_k}})^\top A_{:,\mathcal{J}_{j_k}}\,r_k^{(z)}\|_2^{2}.
\]
Noting that
\(
\langle r_k^{(z)},(A_{:,\mathcal{J}_{j_k}})^\top A_{:,\mathcal{J}_{j_k}}\,r_k^{(z)}\rangle
=\|A_{:,\mathcal{J}_{j_k}}\,r_k^{(z)}\|_2^{2},
\)
we obtain
\[
\|A_{:,\mathcal{J}_{j_k}}\,r_k^{(z)}\|_2^{4}
\le \|r_k^{(z)}\|_2^{2}\,\|(A_{:,\mathcal{J}_{j_k}})^\top A_{:,\mathcal{J}_{j_k}}\,r_k^{(z)}\|_2^{2}.
\]
Consequently, $\alpha_k^{(z)}\ge \alpha_{\mathrm{ex}}^{(z)}$, i.e., \eqref{eq:adaptive_z} with $\delta_z=1$ provides an upper bound on the exact
line search stepsize and thus yields an extrapolated update. More generally, the parameter $\delta_z$ serves as an explicit knob to control the
degree of extrapolation: choosing $\delta_z<1$ recovers a more conservative stepsize and, in particular, one may tune $\delta_z$ to match the exact line search stepsize, while $\delta_z>1$ leads to even more aggressive extrapolation.

From a practical perspective, such extrapolated steps often result in faster decay of the error in the early stage compared with the more conservative choice, since they take larger steps along the same block-averaged descent direction.
Moreover, the adaptive rule \eqref{eq:adaptive_z} is computationally cheaper than evaluating $\alpha_{\mathrm{ex}}^{(z)}$:
computing \eqref{eq:adaptive_z} only requires $A_{:,\mathcal{J}_{j_k}}r_k^{(z)}$, whereas \eqref{eq:exact_line_search_19} additionally involves
$(A_{:,\mathcal{J}_{j_k}})^\top(A_{:,\mathcal{J}_{j_k}}r_k^{(z)})$, thereby saving two matrix--vector products per iteration. The same discussion applies to the $x$-block. We illustrate this phenomenon on a toy example under exactly the same data generation protocol as in Section~\ref{sec:numerical experiments}.
\begin{figure}[!htbp]
\centering

\subfigure[relative error versus iterations]{
\begin{minipage}[t]{0.47\linewidth}\centering
\includegraphics[width=\linewidth]{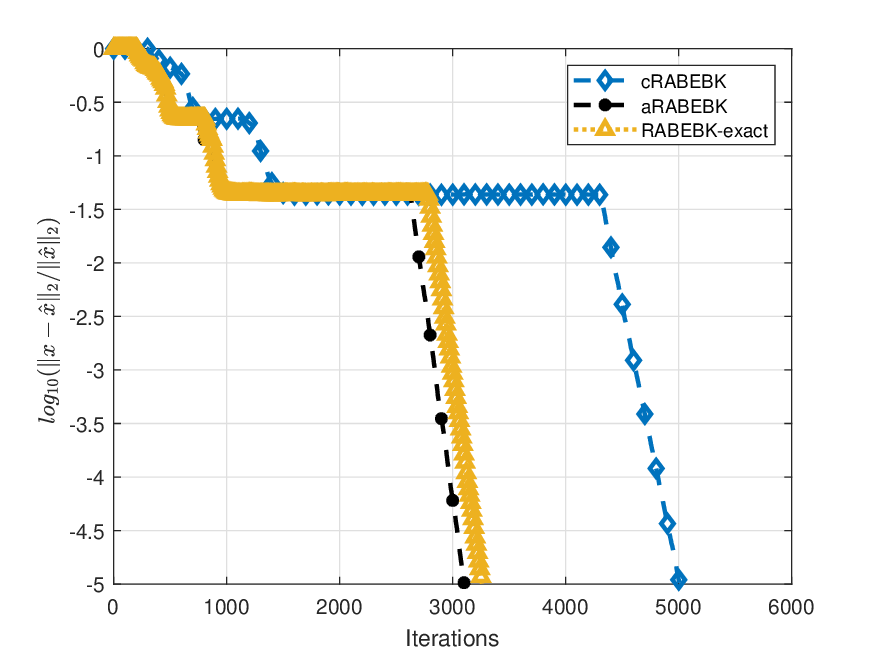}
\end{minipage}}
\hspace{0.5em}
\subfigure[relative error versus CPU time]{
\begin{minipage}[t]{0.47\linewidth}\centering
\includegraphics[width=\linewidth]{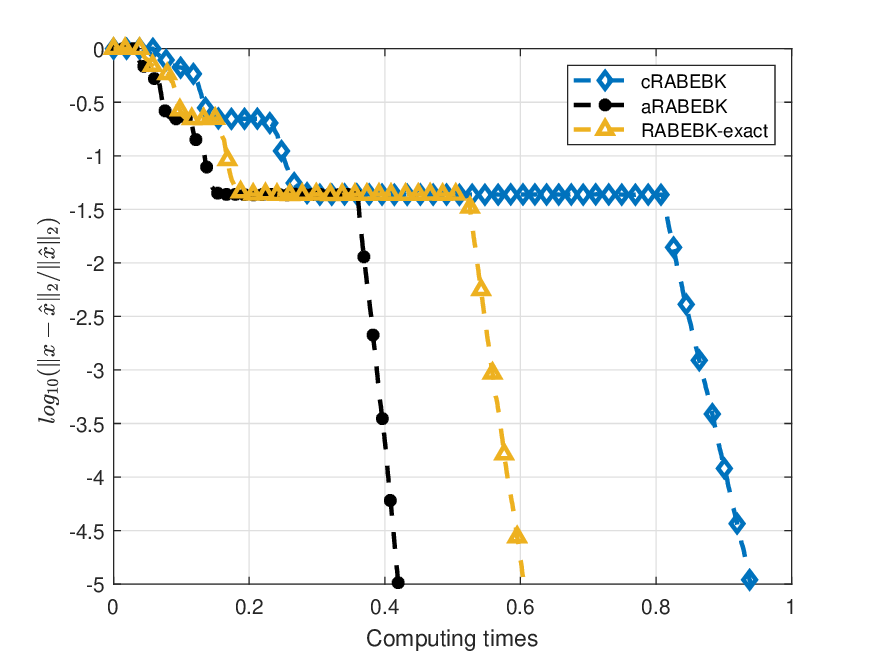}
\end{minipage}}
%\hspace{0.5em}

\subfigure[relaxation parameters on the $z$-block]{
\begin{minipage}[t]{0.47\linewidth}\centering
\includegraphics[width=\linewidth]{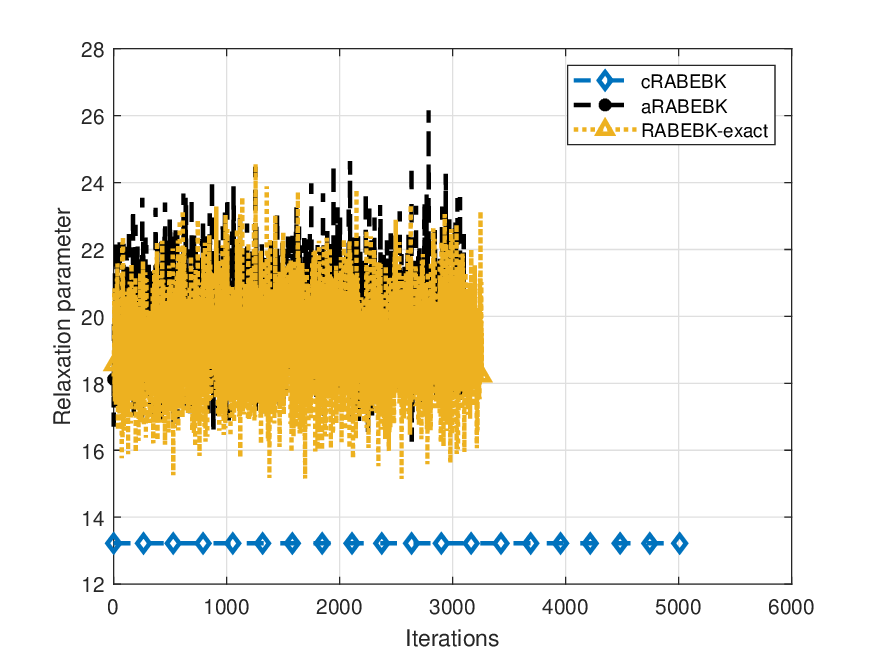}
\end{minipage}}
\hspace{0.5em}
\subfigure[relaxation parameters on the $x$-block]{
\begin{minipage}[t]{0.47\linewidth}\centering
\includegraphics[width=\linewidth]{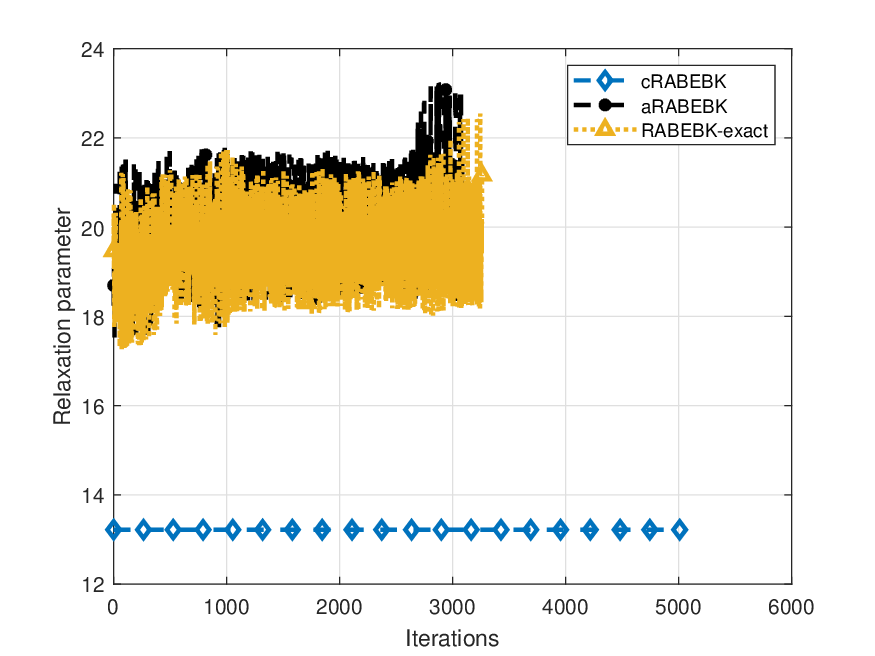}
\end{minipage}}

\caption{Gaussian sparse least-squares test under the same setting as Section~\ref{sec:numerical experiments}:
$A\in\mathbb{R}^{500\times 1000}$ with i.i.d.\ Gaussian entries; the ground-truth $\hat x$ has
$s=\lceil 0.01n\rceil$ nonzeros and $\lambda=5$.}
\label{fig:four}
\end{figure}

Figure~\ref{fig:four} compares the constant, adaptive, and exact line search relaxation strategies under the Gaussian sparse least-squares setting.
Panels (a) and (b) show that the proposed adaptive strategy converges faster than the constant-relaxation baseline and the exact line search variant,
both in terms of iteration counts and CPU time.
Panels (c) and (d) further reveal the mechanism behind this improvement: the adaptive relaxation parameters are consistently larger than the exact line search ones, whereas the constant relaxation remains relatively conservative.
This behavior is observed for both the auxiliary ($z$-block) and primary ($x$-block) updates, which explains the accelerated error decay of the adaptive method.

Based on the above discussion and the proposed adaptive relaxation strategy, we summarize the resulting \emph{adaptive randomized averaging block extended Bregman-Kaczmarz} (aRABEBK) method in Algorithm \ref{alg:aRABEBK}.

\begin{algorithm}[!htbp]
 \caption{A Randomized Averaging Block Extended Bregman-Kaczmarz Method with Adaptive
Relaxation (aRABEBK)} \label{alg:aRABEBK}
 	\begin{algorithmic}[1]
  \Require matrix $A\in \mathbb{R}^{m\times n}$, $b\in \mathbb{R}^m$ and the maximum iteration $T$. Set $x^{*}_{0}=0$, $x _ { 0} = \nabla f^{*}(x^{*}_{0})$, $z_0^{\ast}=b$ and $z_0=\nabla g^*(z_0^*)$. Select the adaptive relaxation parameter $\alpha^{(z)}_k$ and $\alpha^{(x)}_k$.
    \Ensure (approximate) solution of \par $\min_{x \in \mathbb{R}^n} f(x) \quad \text{s.t.} \quad A x = \hat{y}, \quad \text{where} \quad \hat{y} = \arg \min_{y \in \mathcal{R}(A)} g^*(b - y).$
      % \State Initialize $z_0^{\ast}=b$,$z_0=\nabla g^*(z_0^*)$, and  $x_0=x_0^{\ast}=0$.
  \State let $\left\{\mathcal{I}_1,\mathcal{I}_2,\cdots,\mathcal{I}_s \right\}$ and $\left\{\mathcal{J}_1,\mathcal{J}_2,\cdots,\mathcal{J}_t \right\}$
		be partitions of $[m]$ and $[n]$, respectively.
		 \For {$k=0, 1,\ldots,T $ }  
   \State select $j_k\in [t]$ according to probability ${\rm Pr(index}= j_k)= \frac{\|A_{:,\mathcal{J}_{j_k}}\|_F^2}{\|A\|_F^2}$ 
		\State $z_{k+1}^*= z_{k}^* - \delta_z \frac{ \|(A_{:,\mathcal{J}_{j_k}})^\top z_{k}\|_2^2 }{ \|A_{:,\mathcal{J}_{j_k}}(A_{:,\mathcal{J}_{j_k}})^\top z_{k}\|_2^2} A_{:,\mathcal{J}_{j_k}}(A_{:,\mathcal{J}_{j_k}})^\top z_{k}$
  \State  update $z_{k+1}=\nabla g^*(z_{k+1}^*)$
	   \State select $i_k\in [s]$ according to probability ${\rm Pr(index}= i_k)= \frac{\|A_{\mathcal{I}_{i_k},:}\|_F^2}{\|A\|_F^2}$ 
 \State set $ x_{k+1}^{\ast}=x_{k}^{\ast}- \delta_x \frac{\|b_{\mathcal{I}_{i_k}}-A_{\mathcal{I}_{i_k}}x_{k} -z_{k+1,\mathcal{I}_{i_k}}^*\|_2^2 }{\| (A_{\mathcal{I}_{i_k},:})^\top (b_{\mathcal{I}_{i_k}}-A_{\mathcal{I}_{i_k},:}x_{k} -z_{k+1,\mathcal{I}_{i_k}}^*)\|^2_2}(A_{\mathcal{I}_{i_k},:})^\top(A_{\mathcal{I}_{i_k},:}x_k-b_{\mathcal{I}_{i_k}}+z^*_{k+1,\mathcal{I}_{i_k}})$
    \State update  $ x_{k+1}=\nabla f^{\ast}( x_{k+1}^{\ast} )$.
    \State if a stopping criterion is satisfied, stop and \text{return} $x_{k+1}$.
 		\EndFor 
 	\end{algorithmic}
 \end{algorithm}
The adaptive relaxation parameters play a key role in the practical performance observed in our
numerical experiments, and their theoretical implications are analyzed in detail in the subsequent convergence section.

\section{Convergence Analysis}\label{sec:RABEBK-Convergence}
In this section, we investigate the convergence properties of the proposed aRABEBK method. Before presenting the main convergence results, we briefly recall several notions that will be used in the analysis, including \emph{calmness}, \emph{linear regularity}, and \emph{linear growth}, for
completeness; see, e.g.,~\cite{Schpfer2022ExtendedRK}. These properties provide a convenient framework for establishing linear convergence rates.

It is worth noting that the above regularity conditions are satisfied in a wide range of relevant settings, including the case where the function $f$ has a Lipschitz-continuous gradient, as well as when $f$ is a convex piecewise linear-quadratic function, such as
$$
f(x)=\lambda\|x\|_1+\tfrac{1}{2}\|x\|_2^2.
$$

\begin{definition} \label{def:2.2}
The (set-valued) subdifferential mapping $\partial f: \mathbb{R}^n \rightrightarrows \mathbb{R}^n$ is calm at $\hat{x}$ if there are constant $c,L>0$ such that 
\begin{equation*} \label{eq:calm}
    \partial f\,(x)\subset \partial f\,(\hat{x})+L \cdot \|x- \hat{x}\|_{2}\cdot B_{2}\quad \mathrm{for} \ \mathrm{any}\ x\ \mathrm{with}\  \|x- \hat{x}\|_{2}\leq c,
\end{equation*}
where $B_2$ denotes the closed unit ball of the $\ell_2$-norm.
\end{definition}

\begin{definition} \label{def:2.3}
    Let $\partial f(x)\cap \mathcal{R}(A^{\top})\neq \emptyset$. Then the collection $\{\partial f(\hat{x}), \mathcal{R}(A^{\top})\}$ is linearly regular, if there is a constant $\zeta>0$ such that for all $x^* \in \mathbb{R}^n$ we have 
\begin{equation*} \label{eq:linearly regular}
    \mathrm{dist}(x^{*}, \partial f(\hat{x})\cap \mathcal{R}(A^{\top}))\leq \zeta \cdot\left(\mathrm{dist}(x^{*}, \partial f(\hat{x}))+\mathrm{dist}(x^{*},\mathcal{R}(A^{\top}))\right).
\end{equation*}
\end{definition}

\begin{definition} \label{def:2.4}
    The subdifferential mapping of $f$ grows at most linearly, if there exists $\rho_1,\rho_2 \geq 0$ such that for all $x\in \mathbb{R}^n$ and $x^* \in \partial f(x)$ we have 
\begin{equation*} \label{eq:grows at most linearly}
    \| x ^ { * } \|_ 2 \leq \rho _ { 1 } \cdot \| x \| _ { 2 } + \rho _ { 2 }.
\end{equation*}
\end{definition}

\begin{theorem}[\cite{Schpfer2022ExtendedRK}, Theorem~3.9]\label{thm:theta}
    Consider the problem \eqref{eq:BK}, where the function $f:\mathbb{R}^n \rightarrow \mathbb{R}$  is strongly convex. If its subdifferential mapping grows at most linearly, is calm at the unique solution $\hat{x}$ of \eqref{eq:BK} and the collection $\{\partial f(\hat{x}), \mathcal{R}(A^{\top})\}$ is linearly regular, then there exists a constant $\theta(\hat{x})>0$ such that for all $x\in \mathbb{R}^n$ and $x^* \in  \partial f(x) \cap \mathcal{R}(A^{\top})$, the following global error bound holds:
\begin{equation} \label{eq:global error bound}
    D_{f}^{x^{*}}(x, \hat{x})\leq \frac{1}{\theta(\hat{x})}\cdot \|Ax-b\|_{2}^{2}.
\end{equation}
\end{theorem}

\begin{remark}
The global error bound \eqref{eq:global error bound} provides a key link between the residual norm $\|Ax-b\|_2$ and the Bregman distance $D_{f}^{x^{*}}(x,\hat{x})$. In the context of the aRABEBK method, this implies that the adaptive relaxation parameters defined in \eqref{eq:adaptive_z}--\eqref{eq:adaptive_x}, which scale each update according to the local block residual, induce a corresponding reduction in the Bregman distance in the analysis.  
\end{remark}

We now specialize the above result to a commonly used objective function. When  $$f(x) = \lambda \|x\|_1 + \frac{1}{2}\|x\|^2_2,$$ it is immediate that $f$ is $1$-strongly convex and that $\nabla f^*(x^*) = S_{\lambda}(x^*)$, where $S_{\lambda}$ denotes the componentwise soft-thresholding operator. Moreover, as explicitly shown in \cite{schopfer2019linear}, the constant $1/\theta(\hat{x})$ in Theorem~\ref{thm:theta} admits the closed form
%  $\frac{1}{\theta(\hat{x})}$ in Theorem \ref{thm:theta} is 
   \begin{equation} \label{eq:theta better form}
    \frac{1}{\theta(\hat{x})} = \frac { 1 } { \tilde { \sigma } _ { \min } ^ 2  ( A ) } \cdot \frac {  |\hat {  x } |_ {\min} + 2 \lambda } { | \hat { x } | _{\min} },
\end{equation}
where $| \hat { x } | _ { \min } = \min \left\{ | \hat { x } _ { j } | \,|\, j \in \mathrm{supp} ( \hat { x } ) \right\}$ and $\tilde { \sigma } _ { \min } ( A )= \min \left\{ \sigma _ { \min } ( A_{:,\mathcal{J}} ) |\, \, \mathcal{J}\subseteq [n] , A_{:,\mathcal{J}} \right.$ $\left. \neq 0 \right\}$.

Since Algorithm~\ref{alg:aRABEBK} alternates between updating the auxiliary variable $z$ and the
primary variable $x$, we begin by analyzing the convergence properties of the auxiliary sequence $\{z_k^*\}$ generated by aRABEBK. This intermediate analysis is crucial, as it provides the foundation for establishing the convergence behaviour of the primary iterates $\{x_k\}$.

\begin{theorem} \label{thm:convergence RABEBK_adaptive}
 Let \( g : \mathbb{R}^m \to \mathbb{R} \) be $\mu_g$-strongly convex with a Lipschitz-continuous gradient $L_g$. Suppose that the scaling parameter $\delta_z$ satisfies $0<\delta_z<2\mu_g$. Then the auxiliary iterates $\{z_k^*\}$ generated by the aRABEBK method converge in expectation to \[ \hat{z}: = b - \hat{y}, \qquad \text{where} \  \hat{y} \in \mathcal{R}(A), \] which is the unique solution of the constrained optimization problem \eqref{eq:BK01}. In particular, the iterates \( \{z_k^*\} \) satisfy the following estimate:
\begin{equation}\label{eq:convergence of zk_general_ada}
    \mathbb{E} \left[ \| z^*_{k+1} - (b - \hat{y}) \|_2^2 \right] \leq \frac{2L_g}{\mu_g}\left(1-\frac{\delta_z (2\mu_g-\delta_z)\theta(\hat{z})}{2\mu_g\beta_{\max}^{\mathcal{J}}\|A\|^2_F}\right)^{k+1}D_{h}^{\tilde{z}^{*}_0}(z_0, \hat{z}).
\end{equation}
\end{theorem}

\begin{proof}
Computing the approximate solutions $\{z_k^*\}$ amounts to applying the randomized averaging block extended Bregman-Kaczmarz method with the adaptive relaxation parameter $\alpha_k^{(z)}$ to the constrained optimization problem \eqref{eq:BK01}. Therefore, similar to the argument in subsection \ref{sec:RABEBK-d}, we can conclude that the update scheme in aRABEBK for solving $z^*_k $ is equivalent to the scheme
\begin{equation} \label{eq:th421}
    \begin{split}
        &\tilde{z}^*_{k+1} = \tilde{z}^*_k - \frac{\alpha_k^{(z)}}{\|A_{:,\mathcal{J}_{j_k}}\|^2_F} A_{:,\mathcal{J}_{j_k}}(A_{:,\mathcal{J}_{j_k}})^\top z_k, \\
        &z_{k+1} = \nabla h^*(\tilde{z}^*_{k+1}),
    \end{split}
 \end{equation}
with  initialization \( \tilde{z}^*_0 = 0 \). 

Since $g$ is $\mu_g$-strongly convex,  it follows that $h(z)=g(z)-b^\top z$ is also $\mu_g$-strongly convex, and $\nabla h^*$ is $1/\mu_g$-Lipschitz continuous. By the definition of Bregman distance \eqref{eq:Bregman distance} and the standard inequality for strongly convex functions \eqref{eq:alpha-strongly convex}, we have
\begin{equation*}
\begin{split}
    D_{h}^{\tilde{z}^*_{k+1}}(z_{k+1},\hat{z})=& h^{*}(\tilde{z}^*_{k+1})- \langle \tilde{z}^*_{k+1},\hat{z} \rangle +h(\hat{z}) \\
    \leq &h^{*}(\tilde{z}^*_{k}) +\langle \nabla h^{*}(\tilde{z}^*_{k}),\tilde{z}^*_{k+1}-\tilde{z}_k^*\rangle \\
    &+ \frac{1}{2\mu_g}\|\tilde{z}^*_{k+1}-\tilde{z}_k^*  \|^2_2 +h(\hat{z})-\langle \tilde{z}^*_{k+1},\hat{z}\rangle,
\end{split}
\end{equation*}
where $\hat{z}$ is the unique solution of \eqref{eq:BK01}.

Using the fact that $z_k = \nabla h^{*}(\tilde{z}^*_k)$, it is concluded that
\begin{equation*}
    \begin{split}
        &D_{h}^{\tilde{z}_{k+1}^*}(z_{k+1},\hat{z}) \leq D_{h}^{\tilde{z}_k^*}(z_k,\hat{z}) + \langle z_k, \tilde{z}_{k+1}^* - \tilde{z}_k^* \rangle + \frac{1}{2\mu_g}\|\tilde{z}_{k+1}^* - \tilde{z}_k^*\|_2^2 - \langle \tilde{z}_{k+1}^* - \tilde{z}_k^*, \hat{z} \rangle \\
        &= D_{h}^{\tilde{z}_k^*}(z_k,\hat{z}) + \langle \tilde{z}_{k+1}^* - \tilde{z}_k^*, z_k - \hat{z} \rangle + \frac{1}{2\mu_g}\|\tilde{z}_{k+1}^* - \tilde{z}_k^*\|_2^2 \\
        &= D_{h}^{\tilde{z}_k^*}(z_k,\hat{z}) - \langle \frac{\alpha_k^{(z)}}{\|A_{:,\mathcal{J}_{j_k}}\|_F^2} A_{:,\mathcal{J}_{j_k}} (A_{:,\mathcal{J}_{j_k}})^\top z_k, z_k - \hat{z}\rangle + \frac{1}{2\mu_g}\|\tilde{z}_{k+1}^* - \tilde{z}_k^*\|_2^2 \\
        &= D_{h}^{\tilde{z}_k^*}(z_k,\hat{z}) - \langle \frac{\alpha_k^{(z)}(A_{:,\mathcal{J}_{j_k}})^\top z_k}{\|A_{:,\mathcal{J}_{j_k}}\|_F^2}, (A_{:,\mathcal{J}_{j_k}})^\top z_k - (A_{:,\mathcal{J}_{j_k}})^{\top}\hat{z}\rangle + \frac{1}{2\mu_g}\|\tilde{z}_{k+1}^* - \tilde{z}_k^*\|_2^2.
    \end{split}
\end{equation*}
Recall that $(A_{:,\mathcal{J}_{j_k}})^{\top}\hat{z}=0$, which yields
\begin{align*}
D_{h}^{\tilde{z}_{k+1}^*}(z_{k+1},\hat{z}) \leq& D_{h}^{\tilde{z}_k^*}(z_k,\hat{z}) - \frac{\alpha_k^{(z)}}{\|A_{:,\mathcal{J}_{j_k}}\|_F^2} \|(A_{:,\mathcal{J}_{j_k}})^\top z_k\|_2^2 \\
&+ \frac{1}{2\mu_g}\frac{(\alpha_k^{(z)})^2}{\|(A_{:,\mathcal{J}_{j_k}})^\top z_k\|_2^4} \|A_{:,\mathcal{J}_{j_k}} (A_{:,\mathcal{J}_{j_k}})^\top z_k\|_2^2.
\end{align*}
Then, by using the definition  \eqref{eq:adaptive_z} of  $\alpha_k^{(z)}$, we obtain
\begin{align}
  D_{h}^{\tilde{z}_{k+1}^*}(z_{k+1},\hat{z}) 
    & \leq D_{h}^{\tilde{z}_k^*}(z_k,\hat{z}) -\delta_z \frac{\|(A_{:,\mathcal{J}_{j_k}})^\top z_k\|_2^4}{\|A_{:,\mathcal{J}_{j_k}} (A_{:,\mathcal{J}_{j_k}})^\top z_k\|_2^2}+\frac{1}{2\mu_g}\delta_z^2 \frac{\|(A_{:,\mathcal{J}_{j_k}})^\top z_k\|_2^4}{\|A_{:,\mathcal{J}_{j_k}} (A_{:,\mathcal{J}_{j_k}})^\top z_k\|_2^2}\notag \\
    &= D_{h}^{\tilde{z}_k^*}(z_k,\hat{z}) -\frac{\delta_z (2\mu_g-\delta_z)}{2\mu_g} \frac{\|(A_{:,\mathcal{J}_{j_k}})^\top z_k\|_2^4}{\|A_{:,\mathcal{J}_{j_k}} (A_{:,\mathcal{J}_{j_k}})^\top z_k\|_2^2}\notag\\
    &\leq D_{h}^{\tilde{z}_k^*}(z_k,\hat{z})-\frac{\delta_z (2\mu_g-\delta_z)}{2\mu_g}\frac{1}{\sigma_{\max}^2 (A_{:,\mathcal{J}_{j_k}})}\|(A_{:,\mathcal{J}_{j_k}})^\top z_k\|_2^2.\label{eq:conver_z_const6}
\end{align}
% It is obvious that the second term in the last equality reaches its maximum when $\delta_z = \mu_g$, so we set $\delta_z = \mu_g$ in the following proof.
By taking the expectation conditional on the first $k$ iterations on both sides of \eqref{eq:conver_z_const6}
and using the linearity of the expectation, we have
\begin{align}
  \mathbb{E}\left[ D_{h}^{\tilde{z}^{*}_{k+1}}(z_{k+1}, \hat{z}) \right] &\leq \mathbb{E}\left[ D_{h}^{\tilde{z}^{*}_k}(z_k, \hat{z})  \right]- \frac{\delta_z (2\mu_g-\delta_z)}{2\mu_g\|A\|^2_F}\sum_{j_k=1}^{t}\frac{\|A_{:,\mathcal{J}_{j_k}}\|_F^2}{\sigma_{\max}^2 (A_{:,\mathcal{J}_{j_k}})}\|(A_{:,\mathcal{J}_{j_k}})^\top z_k\|_2^2\notag \\
        &\leq \mathbb{E}\left[ D_{h}^{\tilde{z}^{*}_k}(z_k, \hat{z})  \right]- \frac{\delta_z (2\mu_g-\delta_z)}{2\mu_g\|A\|^2_F}\sum_{j_k=1}^{t}\frac{1}{\beta_{\max}^{\mathcal{J}}}\|(A_{:,\mathcal{J}_{j_k}})^\top z_k\|_2^2 \notag \\
        &= \mathbb{E}\left[ D_{h}^{\tilde{z}^{*}_k}(z_k, \hat{z})  \right]- \frac{\delta_z (2\mu_g-\delta_z)}{2\mu_g\|A\|^2_F}\frac{1}{\beta_{\max}^{\mathcal{J}}}\|A^\top z_k\|_2^2.\label{eq:conver_z_const7}
\end{align}
It follows from the inequality \eqref{eq:global error bound} in Theorem \ref{thm:theta} that $$\|A^\top z_k\|_2^2 \geq \theta(\hat{z}) D_{h}^{z_{k}^{*}}(z_{k},\hat{z}).$$ 
Taking the expectation on both sides of \eqref{eq:conver_z_const7} then gives
\begin{equation*}
    \mathbb{E}\left[ D_{h}^{\tilde{z}^{*}_{k+1}}(z_{k+1}, \hat{z}) \right] \leq \left(1-\frac{\delta_z (2\mu_g-\delta_z)\theta(\hat{z})}{2\mu_g\beta_{\max}^{\mathcal{J}}\|A\|^2_F}\right)\mathbb{E}\left[ D_{h}^{\tilde{z}^{*}_k}(z_k, \hat{z})  \right].
\end{equation*}
Lemma \ref{lem:Euclid-Bregman}  says that the Euclidean distance is bounded above by the Bregman distance, thus we have
\begin{equation}
    \mathbb{E} \left[\|z_{k+1}-\hat{z}\|_2^2\right]\leq \frac{2}{\mu_g}\left(1-\frac{\delta_z (2\mu_g-\delta_z)\theta(\hat{z})}{2\mu_g\beta_{\max}^{\mathcal{J}}\|A\|^2_F}\right)^{k+1}D_{h}^{\tilde{z}^{*}_0}(z_0, \hat{z}).
\end{equation}

We now arrive at the desired conclusion:
\begin{equation*} 
    \begin{split}
         \mathbb{E} \left[\|z^*_{k+1}-(b-\hat{y})\|_2^2\right] &=  \mathbb{E} \left[\|\nabla g(z_{k+1})-\nabla g(\hat{z})\|_2^2\right] \\
          &\leq L_g  \mathbb{E} \left[\|z_{k+1}-\hat{z}\|_2^2\right] \\
         &\leq \frac{2L_g}{\mu_g}\left(1-\frac{\delta_z (2\mu_g-\delta_z)\theta(\hat{z})}{2\mu_g\beta_{\max}^{\mathcal{J}}\|A\|^2_F}\right)^{k+1}D_{h}^{\tilde{z}^{*}_0}(z_0, \hat{z}),
    \end{split}
\end{equation*}
where the first inequality exploits the fact that  $g$ has a Lipschitz-continuous gradient with constant $L_g$.
\end{proof}
\begin{remark} \label{remark:zk_general}
    From the estimate in~\eqref{eq:convergence of zk_general_ada}, we observe that
the relaxation parameter $\delta_z$ directly governs the linear convergence rate of the auxiliary sequence $\{z_k^*\}$. For any choice $0 < \delta_z < 2\mu_g$, the iterates $\{z_k^*\}$ converge in expectation to the dual solution $\hat{z} = b - \hat{y}$ with a linear contraction factor
$$1 - \frac{\delta_z(2\mu_g - \delta_z)\,\theta(\hat{z})}{2\mu_g \beta_{\max}^{\mathcal{J}}\|A\|_F^2}.$$
 This factor is minimized when $\delta_z = \mu_g$. In that case the bound simplifies to
    \begin{equation*}
          \mathbb{E} \left[\|z^*_{k+1}-(b-\hat{y})\|_2^2 \right] 
          \leq \frac{2L_g}{\mu_g}
          \left( 1 - \frac{\mu_g \theta(\hat{z}) }{2\beta_{\max}^{\mathcal{J}}\| A \|_F^2} \right)^{k+1}
          D_{h}^{\tilde{z}^{*}_0}(z_0, \hat{z}).
    \end{equation*}
% Consequently, if the selected column blocks are reasonably well conditioned and the relaxation parameter is chosen appropriately, the auxiliary iterates converge to the optimal solution with the fastest linear rate ensured by the above estimate. In particular, the auxiliary updates act to progressively absorb the ill-conditioned or slowly decaying components of the residual, creating a more favorable structure that will be exploited in the analysis of the primary iterates.
\end{remark}

With the linear convergence of the auxiliary sequence $\{z_k^*\}$ established,
we proceed to analyze the convergence behavior of the primary iterates
$\{x_k\}$.
\begin{theorem}
    Let \( f : \mathbb{R}^n \to \mathbb{R} \) be $\mu_f$-strongly convex and assume the conditions in Theorem \ref{thm:theta} hold. Then, the iterates $\{x_k\}$ generated by the aRABEBK method with adaptive relaxation parameter \eqref{eq:adaptive_x} converge in expectation to the unique solution
$\hat{x}$ of the constrained optimization problem \eqref{eq:BK02}. 

More precisely, for any $\zeta,\varepsilon > 0 $ and any $0<\delta_x<2\mu_f$, there exist positive constants $c_1$ and $c_2$ such that
\begin{align} \label{eq:thm_x_ada_convergence_general}
    \mathbb{E} \left[D_{f}^{x_{k+1}^{*}}(x_{k+1},\hat{x})\right] \leq &\left(1-c_1\cdot\gamma(\hat{x})\right)\mathbb{E}\left[D_{f}^{x_{k}^{*}}(x_{k},\hat{x})\right]\notag\\
    &+c_2\cdot \frac{2L_g}{\mu_g}\left(1-\frac{\delta_z (2\mu_g-\delta_z)\theta(\hat{z})}{2\mu_g\beta_{\max}^{\mathcal{J}}\|A\|^2_F}\right)^{k+1}D_{h}^{\tilde{z}^{*}_0}(z_0, \hat{z}),
\end{align}
where $$c_1=\frac{\delta_x(2\mu_f(1-\zeta)-\delta_x) (1-\frac{1}{\varepsilon})}{2\mu_f\beta_{\max}^\mathcal{I}\|A\|_F^2},
\quad c_2 = \frac{\delta_x}{4\zeta \beta_{\min}^\mathcal{I}\|A\|_F^2}-\frac{\delta_x(2\mu_f(1-\zeta)-\delta_x) (1-\varepsilon)}{2\mu_f\beta_{\max}^\mathcal{I}\|A\|_F^2},$$ 
with $\beta^{\mathcal{I}}_{\min} := \min_{i \in [s]} \frac{\sigma^2_{\min} \left( A_{\mathcal{I}_{i},:} \right)}{\| A_{\mathcal{I}_{i},:} \|^2_F}$.
% $$c_3 = \frac{2L_g}{\mu_g}\left(1-\frac{\delta_z (2\mu_g-\delta_z)\theta(\hat{z})}{2\mu_g\beta_{\max}^{\mathcal{J}}\|A\|^2_F}\right)^{k+1}D_{h}^{\tilde{z}^{*}_0}(z_0, \hat{z}).$$
\end{theorem}
\begin{proof}
 We now study the convergence behaviour of the primary iterates $\{x_k\}$. Let $v_k := A_{\mathcal{I}_{i_k},:} x_k - b_{\mathcal{I}_{i_k}} + z^{*}_{k+1,\mathcal{I}_{i_k}}
$.
 % Recall that $$r_k^{(x)}=b_{\mathcal{I}_{i_k}}-A_{\mathcal{I}_{i_k}}x_{k} -z_{k+1,\mathcal{I}_{i_k}}^*.$$ 
Since $\nabla f^*$ is $1/\mu_f$-Lipschitz continuous, it follows from the definition of the Bregman distance~\eqref{eq:Bregman distance} and inequality~\eqref{eq:alpha-strongly convex} that
\begin{align}
   & D_{f}^{x_{k+1}^{*}}(x_{k+1},\hat{x})=f^{*}(x_{k+1}^{*})- \langle x_{k+1}^{*},\hat{x} \rangle +f(\hat{x}) \notag \\
    &=f^*\left(x_k^*-\frac{\alpha_k^{(x)}}{\|A_{\mathcal{I}_{i_k},:}\|^2_F}(A_{\mathcal{I}_{i_k},:})^\top v_k\right)-\langle x_k^*,\hat{x}\rangle+\left\langle \frac{\alpha_k^{(x)}}{\|A_{\mathcal{I}_{i_k},:}\|^2_F}(A_{\mathcal{I}_{i_k},:})^\top v_k, \hat{x}\right\rangle +f(\hat{x})\notag\\
    &\leq f^{*}(x_{k}^{*}) +\langle \nabla f^{*}(x^*_k),-\frac{\alpha_k^{(x)}}{\|A_{\mathcal{I}_{i_k},:}\|^2_F}(A_{\mathcal{I}_{i_k},:})^\top v_k\rangle + \frac{1}{2\mu_f}\left\|-\frac{\alpha_k^{(x)}}{\|A_{\mathcal{I}_{i_k},:}\|^2_F}(A_{\mathcal{I}_{i_k},:})^\top v_k  \right\|^2_2 \notag\\ 
    &-\langle x^*_{k},\hat{x}\rangle+f(\hat{x})+\left\langle \frac{\alpha_k^{(x)}}{\|A_{\mathcal{I}_{i_k},:}\|^2_F} v_k, A_{\mathcal{I}_{i_k},:}\hat{x}\right\rangle.\label{eq:converge_x_const2}
\end{align}

Using the identities $x_k = \nabla f^{*}(x^*_k)$ and $A_{\mathcal{I}_{i_k}} \hat{x} = \hat{y}_{\mathcal{I}_{i_k}}$, together with the definition \eqref{eq: beta_def} of $\beta^{\mathcal{I}}_{\max}$,
 the inequality~\eqref{eq:converge_x_const2} can be rewritten as
 \begin{align}\label{eq:converge_x_const8}
 D_{f}^{x_{k+1}^{*}}(x_{k+1},\hat{x}) \leq &  D_{f}^{x_{k}^{*}}(x_{k},\hat{x}) - \frac{\alpha_k^{(x)}}{\|A_{\mathcal{I}_{i_k},:}\|^2_F}\left\langle  v_k, A_{\mathcal{I}_{i_k},:}x_k - \hat{y}_{\mathcal{I}_{i_k}}\right\rangle \notag\\
 &+ \frac{1}{2\mu_f}\frac{(\alpha_k^{(x)})^2}{\|A_{\mathcal{I}_{i_k},:}\|^4_F}\|(A_{\mathcal{I}_{i_k},:})^\top v_k\|^2_2. 
\end{align}

Note that for any $\zeta>0$, we can decompose
\begin{equation*}
    \begin{split}
        \left\langle  v_k, A_{\mathcal{I}_{i_k},:}x_k - \hat{y}_{\mathcal{I}_{i_k}}\right\rangle &= \left\langle  v_k, A_{\mathcal{I}_{i_k},:}x_k -  b_{\mathcal{I}_{i_k}} + z^{*}_{k+1,\mathcal{I}_{i_k}}+ b_{\mathcal{I}_{i_k}} -z^{*}_{k+1,\mathcal{I}_{i_k}}-\hat{y}_{\mathcal{I}_{i_k}}\right\rangle \\
        &= \|v_k\|_2^2+\left\langle  v_k, b_{\mathcal{I}_{i_k}} -z^{*}_{k+1,\mathcal{I}_{i_k}}-\hat{y}_{\mathcal{I}_{i_k}}\right\rangle. 
       \end{split}
\end{equation*}
Applying Young’s inequality yields
$$
 \left\langle  v_k, A_{\mathcal{I}_{i_k},:}x_k - \hat{y}_{\mathcal{I}_{i_k}}\right\rangle  \geq \|v_k\|_2^2 - \zeta \|v_k\|_2^2 - \frac{1}{4\zeta}\|b_{\mathcal{I}_{i_k}} -z^{*}_{k+1,\mathcal{I}_{i_k}}-\hat{y}_{\mathcal{I}_{i_k}}\|_2^2.
$$
Therefore, it holds that 
\begin{align} \label{eq:extra1}
    - \frac{\alpha_k^{(x)}}{\|A_{\mathcal{I}_{i_k},:}\|^2_F}\left\langle  v_k, A_{\mathcal{I}_{i_k},:}x_k - \hat{y}_{\mathcal{I}_{i_k}}\right\rangle \leq& - \frac{\alpha_k^{(x)}}{\|A_{\mathcal{I}_{i_k},:}\|^2_F}(1-\zeta)\|v_k\|_2^2 \notag\\
    &+\frac{\alpha_k^{(x)}}{4\|A_{\mathcal{I}_{i_k},:}\|^2_F}\frac{1}{\zeta}\|b_{\mathcal{I}_{i_k}} -z^{*}_{k+1,\mathcal{I}_{i_k}}-\hat{y}_{\mathcal{I}_{i_k}}\|_2^2.
\end{align}
Using the definition \eqref{eq:adaptive_x} of adaptive stepsize $\alpha_k^{(x)}$, we obtain
\begin{equation*}
\begin{split}
    \frac{\alpha_k^{(x)}}{4\|A_{\mathcal{I}_{i_k},:}\|^2_F}\frac{1}{\zeta}\|b_{\mathcal{I}_{i_k}} -z^{*}_{k+1,\mathcal{I}_{i_k}}-\hat{y}_{\mathcal{I}_{i_k}}\|_2^2 &= \frac{\delta_x}{4\zeta}\frac{\|v_k\|^2_2}{\|(A_{\mathcal{I}_{i_k},:})^\top v_k\|^2_2}\|b_{\mathcal{I}_{i_k}} -z^{*}_{k+1,\mathcal{I}_{i_k}}-\hat{y}_{\mathcal{I}_{i_k}}\|_2^2 \\
    &\leq  \frac{\delta_x}{4\zeta \sigma_{\min}^2(A_{\mathcal{I}_{i_k},:})}\|b_{\mathcal{I}_{i_k}} -z^{*}_{k+1,\mathcal{I}_{i_k}}-\hat{y}_{\mathcal{I}_{i_k}}\|_2^2.
\end{split}  
\end{equation*}
Here we used the inequality $\|(A^\top y)\|_2 \ge \sigma_{\min}(A)\|y\|_2$ for all $y\in\mathcal{R}(A)$,
together with the fact that $v_k\in\mathcal{R}(A_{\mathcal{I}_{i_k},:})$.

Thus, the inequality \eqref{eq:converge_x_const8} now can be rewritten in the form
\begin{align}
        D_{f}^{x_{k+1}^{*}}(x_{k+1},\hat{x}) &\leq  D_{f}^{x_{k}^{*}}(x_{k},\hat{x}) - \frac{\alpha_k^{(x)}(1-\zeta)}{\|A_{\mathcal{I}_{i_k},:}\|_F^2}\|v_k\|_2^2+\frac{1}{2\mu_f}\frac{(\alpha_k^{(x)})^2}{\|A_{\mathcal{I}_{i_k},:}\|_F^4}\|(A_{\mathcal{I}_{i_k},:})^\top v_k\|^2_2
        \notag\\&+\frac{\delta_x}{4\zeta \sigma_{\min}^2(A_{\mathcal{I}_{i_k},:})}\|b_{\mathcal{I}_{i_k}} -z^{*}_{k+1,\mathcal{I}_{i_k}}-\hat{y}_{\mathcal{I}_{i_k}}\|_2^2 \notag\\
        &=D_{f}^{x_{k}^{*}}(x_{k},\hat{x}) - \delta_x \frac{(1-\zeta)\|v_k\|_2^4}{\|(A_{\mathcal{I}_{i_k},:})^\top v_k\|^2_2}+\frac{1}{2\mu_f}\delta^2_x \frac{\|v_k\|_2^4}{\|(A_{\mathcal{I}_{i_k},:})^\top v_k\|^2_2}\notag\\&+\frac{\delta_x}{4\zeta \sigma_{\min}^2(A_{\mathcal{I}_{i_k},:})}\|b_{\mathcal{I}_{i_k}} -z^{*}_{k+1,\mathcal{I}_{i_k}}-\hat{y}_{\mathcal{I}_{i_k}}\|_2^2 \notag\\
        &=D_{f}^{x_{k}^{*}}(x_{k},\hat{x}) - \frac{\delta_x(2\mu_f(1-\zeta)-\delta_x)}{2\mu_f} \frac{\|v_k\|_2^4}{\|(A_{\mathcal{I}_{i_k},:})^\top v_k\|^2_2} \notag\\ &+\frac{\delta_x}{4\zeta \sigma_{\min}^2(A_{\mathcal{I}_{i_k},:})}\|b_{\mathcal{I}_{i_k}} -z^{*}_{k+1,\mathcal{I}_{i_k}}-\hat{y}_{\mathcal{I}_{i_k}}\|_2^2. \notag\\
        &\leq  D_{f}^{x_{k}^{*}}(x_{k},\hat{x}) - \frac{\delta_x(2\mu_f(1-\zeta)-\delta_x)}{2\mu_f}\frac{1}{\sigma_{\max}^2 (A_{\mathcal{I}_{i_k},:})}\|v_k\|_2^2 \notag\\
        &+\frac{\delta_x}{\zeta \sigma_{\min}^2(A_{\mathcal{I}_{i_k},:})}\|b_{\mathcal{I}_{i_k}} -z^{*}_{k+1,\mathcal{I}_{i_k}}-\hat{y}_{\mathcal{I}_{i_k}}\|_2^2. \label{eq:xxxxx}
\end{align}

Inequality \eqref{eq:xxxxx} provides a preliminary descent estimate for the Bregman distance associated with the primary iterate $x_{k+1}$. To further refine this bound, we decompose $\|v_k\|_2^2$ as
\begin{align} \label{eq:convergence_vk}
    \|v_k\|^2_2 =& \|A_{\mathcal{I}_{i_k},:}x_k - \hat{y}_{\mathcal{I}_{i_k}}\|^2_2 + 2 \left\langle \hat{y}_{\mathcal{I}_{i_k}} - b_{\mathcal{I}_{i_k}} + z^{*}_{k+1,\mathcal{I}_{i_k}}, A_{\mathcal{I}_{i_k},:}x_k - \hat{y}_{\mathcal{I}_{i_k}}\right\rangle \notag\\
    &+ \|z^*_{k+1,\mathcal{I}_{{i_k}}}-(b_{\mathcal{I}_{i_k}}-\hat{y}_{\mathcal{I}_{i_k}})\|^2_2.
\end{align}
Moreover, for any $\varepsilon>0$, Young’s inequality implies
\begin{equation} \label{eq:Du2020_2.5}
\begin{split}
    2\left\langle \hat{y}_{\mathcal{I}_{i_k}} - b_{\mathcal{I}_{i_k}} + z^{*}_{k+1,\mathcal{I}_{i_k}}, A_{\mathcal{I}_{i_k},:}x_k - \hat{y}_{\mathcal{I}_{i_k}}\right\rangle \geq &-\varepsilon \|\hat{y}_{\mathcal{I}_{i_k}} - b_{\mathcal{I}_{i_k}} + z^*_{k+1,\mathcal{I}_{{i_k}}}\|^2_2 \\
    &-\frac{1}{\varepsilon}\|A_{\mathcal{I}_{i_k},:}x_k - \hat{y}_{\mathcal{I}_{i_k}}\|_2^2.
\end{split}
\end{equation}
This decomposition allows us to separate the contribution of the primary residual  $A_{\mathcal{I}_{i_k},:}x_k - \hat{y}_{\mathcal{I}_{i_k}}$ from the error associated with the auxiliary iterate $z^*_{k+1,\mathcal{I}_{i_k}}-(b_{\mathcal{I}_{i_k}}-\hat{y}_{\mathcal{I}_{i_k}})$,  which is essential for deriving a sharp contraction estimate for the sequence $\{x_k\}$.

Combining inequality \eqref{eq:xxxxx}, \eqref{eq:convergence_vk} and \eqref{eq:Du2020_2.5}, we obtain
\begin{equation}\label{eq:dzy_x}
\begin{split}
    &D_{f}^{x_{k+1}^{*}}(x_{k+1},\hat{x}) \leq  D_{f}^{x_{k}^{*}}(x_{k},\hat{x}) - \frac{\delta_x(2\mu_f(1-\zeta)-\delta_x)}{2\mu_f}\frac{1-{1}/{\varepsilon}}{\sigma_{\max}^2 (A_{\mathcal{I}_{i_k},:})}\|A_{\mathcal{I}_{i_k},:}x_k - \hat{y}_{\mathcal{I}_{i_k}}\|_2^2 \\
    &+\left(\frac{\delta_x}{4\zeta \sigma_{\min}^2(A_{\mathcal{I}_{i_k},:})}-\frac{\delta_x(2\mu_f(1-\zeta)-\delta_x)}{2\mu_f}\frac{1-\varepsilon}{\sigma_{\max}^2 (A_{\mathcal{I}_{i_k},:})}\right)\|z^*_{k+1,\mathcal{I}_{{i_k}}}-(b_{\mathcal{I}_{i_k}}-\hat{y}_{\mathcal{I}_{i_k}})\|^2_2.
\end{split}
\end{equation}
By taking the expectation conditional on the first $k$ iterations on both sides of \eqref{eq:dzy_x} and using the linearity of the expectation, we have
\begin{equation}  \label{eq:convergence_x_general_last_ada}
    \mathbb{E}\left[D_{f}^{x_{k+1}^{*}}(x_{k+1},\hat{x})\right]\leq \mathbb{E}\left[D_{f}^{x_{k}^{*}}(x_{k},\hat{x})\right] - c_1 \cdot \mathbb{E}\left[\|Ax_k - \hat{y}\|_2^2\right]+c_2\cdot \mathbb{E}\left[\|z_{k+1}^* - (b-\hat{y})\|_2^2\right].
\end{equation}
It follows from the inequality \eqref{eq:global error bound} in Theorem \ref{thm:theta} that $$D_{f}^{x_{k}^{*}}(x_{k},\hat{x})\leq \frac{1}{\gamma(\hat{x})}\|Ax_k-y\|_2^2.$$ 
Combining this estimate with \eqref{eq:convergence_x_general_last_ada} and taking expectations again yields
\begin{equation}
\begin{split}
   & \mathbb{E}\left[D_{f}^{x_{k+1}^{*}}(x_{k+1},\hat{x})\right]\leq \left(1-c_1\cdot\gamma(\hat{x})\right)\mathbb{E}\left[D_{f}^{x_{k}^{*}}(x_{k},\hat{x})\right] +c_2\cdot \mathbb{E}\left[\|z_{k+1}^* - (b-\hat{y})\|^2_2\right] \\
    &\leq \left(1-c_1\cdot\gamma(\hat{x})\right)\mathbb{E}\left[D_{f}^{x_{k}^{*}}(x_{k},\hat{x})\right] +c_2\cdot \frac{2L_g}{\mu_g}\left(1-\frac{\delta_z (2\mu_g-\delta_z)\theta(\hat{z})}{2\mu_g\beta_{\max}^{\mathcal{J}}\|A\|^2_F}\right)^{k+1}D_{h}^{\tilde{z}^{*}_0}(z_0, \hat{z}).
\end{split}
\end{equation}
where the second inequality follows from \eqref{eq:convergence of zk_general_ada}. Since $c_2$ is a constant and the second term vanishes as $k\to\infty$, it follows that the sequence $\{x_k\}$ converges in expectation to the unique solution $\hat{x}$. This completes the proof.
\end{proof}

\begin{remark} \label{remark:xk_general}
Note that the expected linear convergence rate of the primary iterates $\{x_k\}$ with an adaptive relaxation parameter, as stated in \eqref{eq:thm_x_ada_convergence_general}, is given by
 $$1-\frac{\delta_x(2\mu_f(1-\zeta)-\delta_x) (1-\frac{1}{\varepsilon})}{2\mu_f\beta_{\max}^\mathcal{I}\|A\|_F^2}\cdot \gamma(\hat{x}).$$
Since the above bound holds for any $\zeta,\varepsilon>0$, one may take $\varepsilon$ sufficiently large and $\zeta$ sufficiently small so that their influence becomes negligible. In this regime, the dominant term governing the convergence rate is $\delta_x(2\mu_f-\delta_x)$, which attains its maximum when $\delta_x=\mu_f$. Consequently, the leading convergence rate of $\{x_k\}$ reduces
to 
\begin{equation} \label{eq:25-12-12-1}
1-\frac{\mu_f}{2\beta_{\max}^{\mathcal{I}}\|A\|_F^2}\cdot \gamma(\hat{x}).
\end{equation}
\end{remark}

For comparison, the expected linear convergence rate of the REBK method is given by
\begin{equation} \label{eq:25-12-12-2}
1-\frac{\mu_f}{2\|A\|^2_F }\cdot \gamma(\hat{x}).
\end{equation}
Since $\beta_{\max}^{\mathcal{I}}<1$ by definition, a direct comparison of \eqref{eq:25-12-12-1} and \eqref{eq:25-12-12-2} shows that the proposed aRABEBK method admits a strictly faster theoretical convergence rate than REBK. This improvement is mainly attributed to the block-averaging mechanism, which effectively replaces the global  quantity $\|A\|_F^2$ by the smaller block-dependent quantity $\beta_{\max}^{\mathcal{I}}\|A\|_F^2$, leading to a tighter contraction factor.
This comparison highlights the benefit of incorporating block averaging within the extended Bregman-Kaczmarz framework.

On the other hand, although cRABEBK shares the same asymptotic convergence factor as \eqref{eq:25-12-12-1}, this rate is attained only when the constant relaxation parameter is chosen optimally according to the global quantity $\beta_{\max}^{\mathcal{I}}$. 
By contrast, the adaptive relaxation parameters employed in aRABEBK are computed from local residual and blockwise information and are updated dynamically at each iteration. While this adaptivity does not further improve the asymptotic contraction constant, it typically leads to substantially faster error reduction in the regimes where the iterates are still far from the solution.

Therefore, the advantage of aRABEBK over cRABEBK does not lie in the asymptotic rate itself, but in accelerating convergence toward the asymptotic regime and enhancing numerical robustness, as demonstrated by the numerical experiments.

Finally, consider the  problem \eqref{eq:main_partical} with the objective function
\[
f(x)=\lambda\|x\|_1+\frac{1}{2}\|x\|_2^2.
\]
In this case, the quantity \(\gamma(\hat{x})\) appearing in the convergence analysis
can be computed explicitly; see \eqref{eq:theta better form}. In particular, for
the sparse least-squares problem, the proposed method admits an explicit estimate
of the expected linear convergence factor, given by
\[
1 - \frac{\mu_f}{2 \beta^{\mathcal{I}}_{\max} \|A\|_F^2} \cdot \frac{1}{\tilde{\sigma}_{\min}^2(A)} \cdot \frac{|\hat{x}|_{\min} + 2\lambda}{|\hat{x}|_{\min}}.
\]

\section{Numerical experiments} \label{sec:numerical experiments}
In this section, we evaluate the performance of the proposed algorithm on two types of least-squares problems. The first is a sparse least-squares problem, with objective function $$f(x) = \lambda \|x\|_1 + \frac{1}{2}\|x\|_2^2.$$ The second is a classical minimum-norm least-squares problem, for which $$f(x) = \frac{1}{2}\|x\|_2^2.$$ In both cases, the data misfit function is given by $$g^*(b - y) = \frac{1}{2}\|b - y\|_2^2.$$ 

All numerical experiments are performed in \texttt{MATLAB} on a computer with an Intel Core i7 processor and 16GB of RAM.

For the sparse least-squares problem, the ground-truth sparse solution $\hat{x} \in \mathbb{R}^n$ is generated by drawing a random vector with $s$ nonzero entries from the standard normal distribution, where the sparsity level is set to $s = \lceil 0.01\,n\rceil$. The regularization parameter is fixed at $\lambda = 5$. For the minimum-norm least-squares problem, the reference solution $\hat{x}$ is computed using the \texttt{MATLAB} function \texttt{lsqminnorm}.

In both settings, the right-hand side is constructed as $$ b = A\hat{x} + e, $$
where $e$ denotes the noise vector. Following standard practice in minimum-norm and constrained least-squares problems, the noise is chosen to lie in $\mathrm{null}(A^\top)$ so that it does not affect the normal equations and therefore preserves the same least-squares solution set. Specifically, we generate $e = N v$, where the columns of $N$ form an orthonormal basis of
$\mathrm{null}(A^\top)$, generated using the \texttt{MATLAB} function \texttt{null}. The vector $v$ is drawn uniformly at random from the sphere $\partial B_\rho(0)$ with radius $\rho = q \|A\hat{x}\|_2$, where the parameter $q$ controls the relative noise level and is fixed to $q = 5$ throughout all experiments.

In the experiments, we evaluate the numerical performance of the following four state-of-the-art Kaczmarz-type methods:
\begin{enumerate}
   \item \textbf{REBK\cite{Schpfer2022ExtendedRK}}: randomized extended Bregman-Kaczmarz method.
    \item \textbf{REABK}\cite{du2020randomized}: randomized extended averaging block Kaczmarz method.
 \item \textbf{cRABEBK}\cite{dong2025relaxed}: randomized averaging block extended Bregman-Kaczmarz method with a constant relaxation parameter
   $\beta_{\max}^{-1}$.
   \item \textbf{aRABEBK}: adaptive randomized averaging block extended Bregman-Kacz- marz method, corresponding to Algorithm \ref{alg:aRABEBK} with $\delta_{z}=\delta_{x}=1$.
\end{enumerate}

For the sparse least-squares setting, the gradient mappings appearing in Algorithm \ref{alg:aRABEBK} admit closed-form expressions. In particular, since \( \nabla g^*(z_k^*) = z_k^* \), step 5 simplifies to \( z_{k+1} = z_{k+1}^* \). Moreover, step 8 simplifies to
\[
x_{k+1} = S_\lambda(x_{k+1}^*),
\]
where \( S_\lambda(\cdot) \) denotes the soft-thresholding (shrinkage) operator. Since the {REABK} method is designed for smooth least-squares problems and cannot directly handle sparsity-inducing regularization, it is excluded from this comparison. Therefore, in the sparse least-squares experiments, we compare the performance of {REBK}, cRABEBK, and {aRABEBK}.

For the minimum-norm least-squares problem, both step~5 and step~8 in
Algorithm~\ref{alg:aRABEBK} reduce to identity mappings, namely
\( z_{k+1} = z_{k+1}^* \) and \( x_{k+1} = x_{k+1}^* \). In this setting, the
constant-relaxation method {cRABEBK} coincides with {REABK}, which
has been shown to outperform {REBK} in \cite{du2020randomized}. Therefore, we compare
{REABK} and {aRABEBK} in order to isolate and highlight the effect
of adaptive relaxation.

For all algorithms, we initialize $z_0 = b$ and $x_0^* = 0$. Each algorithm is terminated either when the iteration count exceeds $5 \times 10^6$ or when the relative error 
\[
\mathrm{ERR} = \frac{\|x_k - \hat{x}\|_2}{\|\hat{x}\|_2}
\] 
falls below $10^{-5}$.

Our algorithm relies on a block structure. We assume that the row and column subsets $\{\mathcal{I}_i\}_{i=1}^{s}$ and $\{\mathcal{J}_j\}_{j=1}^{t}$ have uniform size $\tau$ wherever possible, i.e., $|\mathcal{I}_i| = |\mathcal{J}_j| = \tau$ for $i = 1,\dots,s-1$ and $j = 1,\dots,t-1$. The row partition is defined as
\begin{equation*}
   \begin{split}
       \mathcal{I}_i &= \{(i-1)\tau + 1,\, (i-1)\tau + 2,\, \ldots,\, i\tau\}, \quad i = 1,2,\ldots,s-1, \\
       \mathcal{I}_s &= \{(s-1)\tau + 1,\, (s-1)\tau + 2,\, \ldots,\, m\}, \qquad |\mathcal{I}_s| \leq \tau,
   \end{split}
\end{equation*}
and the column partition as
\begin{equation*}
   \begin{split}
       \mathcal{J}_j &= \{(j-1)\tau + 1,\, (j-1)\tau + 2,\, \ldots,\, j\tau\}, \quad j = 1,2,\ldots,t-1,\\
       \mathcal{J}_t &= \{(t-1)\tau + 1,\, (t-1)\tau + 2,\, \ldots,\, n\}, \qquad |\mathcal{J}_t| \leq \tau.
   \end{split}
\end{equation*}
In all experiments, the block size $\tau$ for REABK, RABEBK and aRABEBK are set to 20. 

In the cRABEBK method, the relaxation parameters $\alpha_k^{(z)}$ and $\alpha_k^{(x)}$ are fixed and chosen as
\[
\alpha = 1/\beta_{\max}, \qquad 
\beta_{\max} :=\max \left\{\beta^{\mathcal{I}}_{\max},\beta^{\mathcal{J}}_{\max}\right\}.
\]

In contrast, the aRABEBK method employs adaptive relaxation parameters computed from local residual information at each iteration. In all experiments, we fix $\delta_z = \delta_x = 1$, which corresponds to the admissible adaptive step size allowed by the theory. This parameter-free choice was found to be robust across all tested problem instances and eliminates the need for
problem-dependent tuning.

\subsection{Experiments on synthetic data}
In this subsection, we assess the performance of the compared algorithms on two classes of synthetic test matrices $A$, in order to illustrate their behaviour under different problem settings.

\subsubsection{Gaussian matrix}
We begin by considering Gaussian coefficient matrices, generated using the \texttt{MATLAB} command \texttt{randn}.

\begin{table}[!htbp] 
\scriptsize 
\centering 
\caption{ Numerical results for Gaussian matrices in the sparse least-squares case. }\label{table:Guassian}  
% \resizebox{\textwidth}{!}{  
\begin{tabular}{|c|c|c|c|c|c|c|c|c|c|}
\hline
\multirow{2}{*} {$m\times n$}
&\multicolumn{2}{c}{REBK} &\multicolumn{2}{|c}{cRABEBK} &\multicolumn{2}{|c|}{aRABEBK} \\ 
\cline{2-7} 
& IT &CPU  &IT &CPU &IT &CPU \\ 
\hline  
$1000\times 500$   &	 90624&	 3.896&	  6795&	1.340&	 4697&	1.045  \\  
\hline  
$500\times 1000$   &	 55189&	 2.295&	 	 4402&	 0.915&	 2844&	0.389  \\ 
\hline  
$2000\times 1000$  &	 331375&	177.384 &	 	 22066&	20.255&	 15560&	12.592  \\  
\hline  
$1000\times 2000$   &	 698962&	 422.960&	 	 46252&	52.037&	 34254&	31.570  \\  
\hline  
$4000\times 2000$   &	 195094&	 701.151&	 	 11870&	50.731&	 8814&	36.463  \\   
\hline  
$2000\times 4000$   &	 906598&	 2941.993&	 	 55125&	258.776&	 49152&	 231.407 \\   
\hline  
\end{tabular}  
% }   
\end{table} 
Table~\ref{table:Guassian} reports the numerical results for Gaussian matrices of various sizes. Both {cRABEBK} and {aRABEBK} significantly outperform {REBK} in terms of iteration counts and CPU time across all tested dimensions, which can be attributed to the use of block averaging that exploits information from multiple hyperplanes at each iteration.

Moreover, the adaptive variant {aRABEBK} consistently achieves the best performance among all methods. For example, when \( m \times n = 2000 \times 4000 \), {aRABEBK} is approximately {12.7$\times$ faster} than {REBK} in CPU time, reducing the runtime from 2941.99 seconds to 231.41 seconds. These results demonstrate that adaptive relaxation provides a further and nontrivial efficiency gain beyond that achieved by block averaging alone.

Comparing {aRABEBK} with its constant-relaxation counterpart {cRABEBK}, we observe that adaptive relaxation yields a further and non-negligible improvement. In the same test, {aRABEBK} reduces the iteration count from 55,125 to 49,152 and achieves about a {1.12$\times$ speedup} in CPU time (231.41 seconds versus 258.78 seconds). This consistent advantage across all tested problem sizes demonstrates that adaptively chosen relaxation parameters can effectively accelerate convergence beyond what is achievable with a fixed relaxation strategy.
\begin{figure}[!htbp] 
\centering
	\subfigure[$m=1000,n=500$]   
	{
		\begin{minipage}[t]{0.31\linewidth}
			\centering
			\includegraphics[width=1\textwidth]{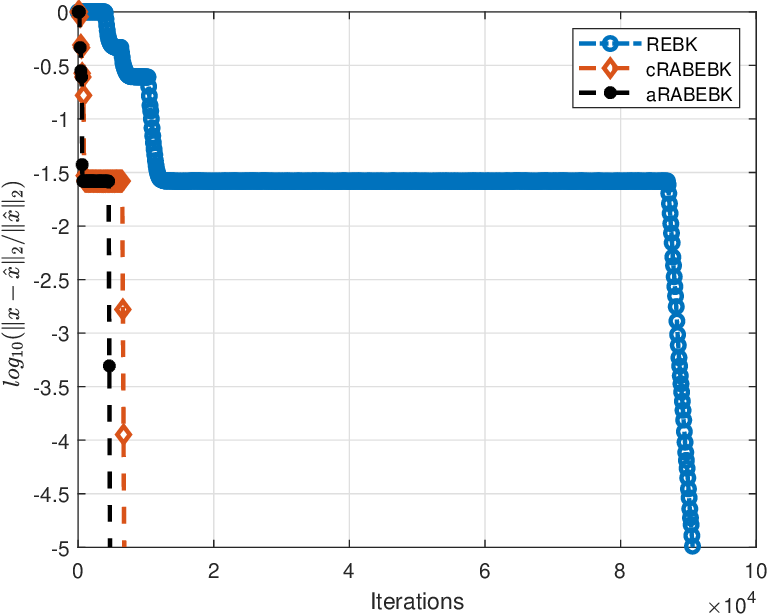}
		\end{minipage}
	}
	\subfigure[$m=2000,n=1000$]  
	{
		\begin{minipage}[t]{0.31\linewidth}
			\centering
			\includegraphics[width=1\textwidth]{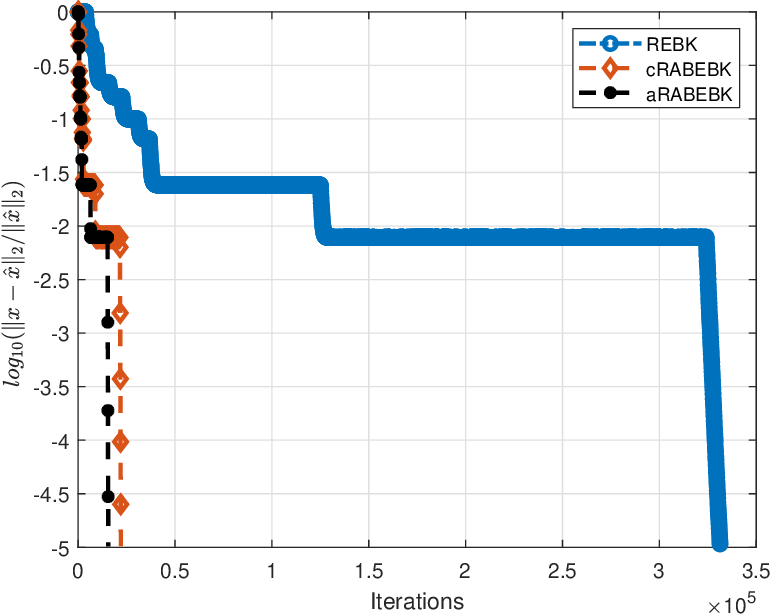}
		\end{minipage}
      }
    \subfigure[$m=4000,n=2000$]   
	{
		\begin{minipage}[t]{0.31\linewidth}
			\centering
			\includegraphics[width=1\textwidth]{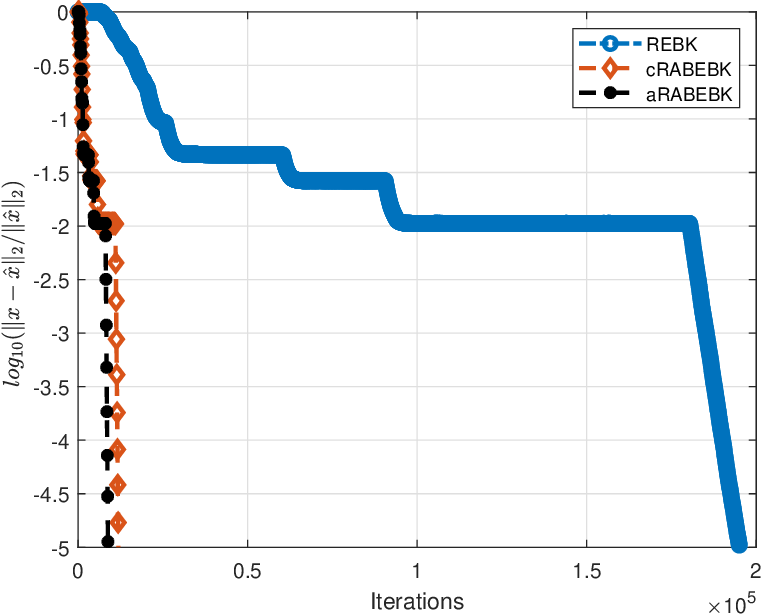}
		\end{minipage}
	}
  
	\subfigure[$m=500,n=1000$]   
	{
		\begin{minipage}[t]{0.31\linewidth}
			\centering
			\includegraphics[width=1\textwidth]{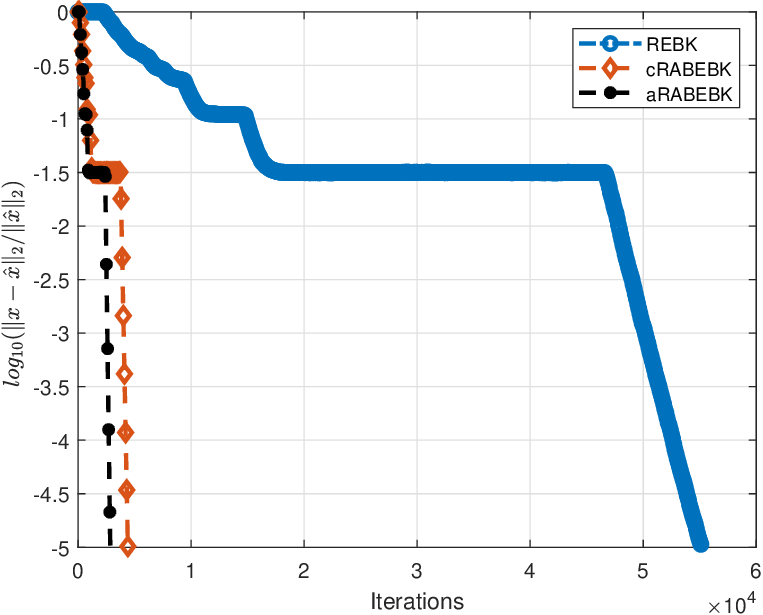}
		\end{minipage}
	}
	\subfigure[$m=1000,n=2000$]  
	{
		\begin{minipage}[t]{0.31\linewidth}
			\centering
			\includegraphics[width=1\textwidth]{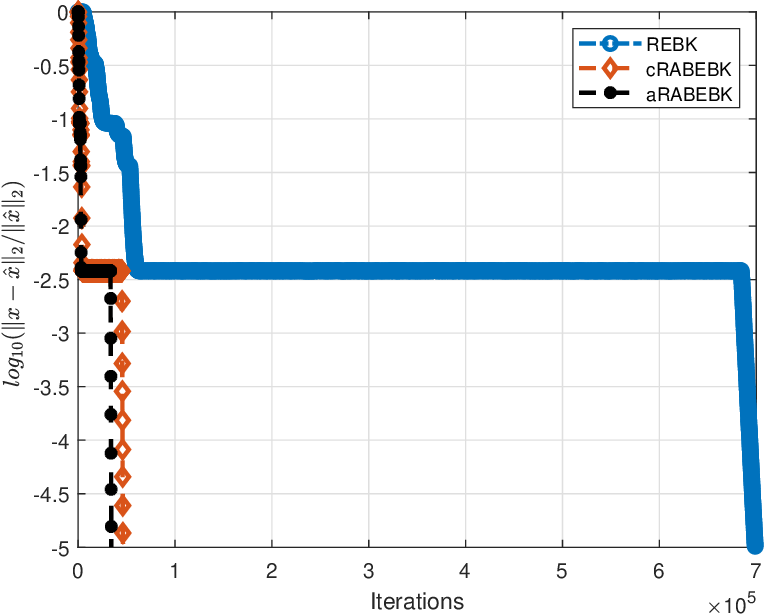}
		\end{minipage}
      }
    \subfigure[$m=2000,n=4000$]   
	{
		\begin{minipage}[t]{0.31\linewidth}
			\centering
			\includegraphics[width=1\textwidth]{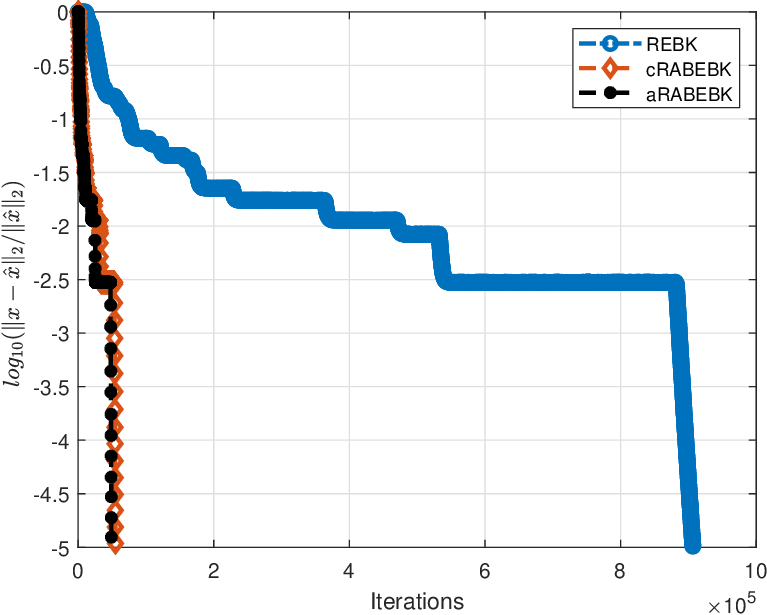}
		\end{minipage}
	}

\caption{Curves of the relative error versus the number of iterations for the sparse least-squares case with Gaussian matrices. Top: overdetermined matrices; Bottom: underdetermined matrices.}
 \label{fig:Guassian}
\end{figure}

Figure~\ref{fig:Guassian} illustrates the convergence behavior of the REBK, cRABEBK, and aRABEBK methods for the sparse least-squares problem with Gaussian matrices of six different sizes. The plots show the relative error versus the number of iterations. It is evident that aRABEBK (black line) converges substantially faster than both REBK and cRABEBK across all tested problem dimensions, highlighting the effectiveness of adaptive relaxation in accelerating convergence.

We now report the numerical performance of the REABK and aRABEBK methods for Gaussian coefficient matrices in the minimum-norm least-squares setting.

\begin{table}[!htbp] 
% \scriptsize 
\centering 
\caption{Numerical results for Gaussian matrices in minimum-norm least-squares case.}\label{table:mini_Gaussian}  
% \resizebox{\textwidth}{!}{  
\begin{tabular}{|c|c|c|c|c|c|}
\hline
\multirow{2}{*}{$m \times n$}
& \multicolumn{2}{c|}{REABK} 
& \multicolumn{2}{c|}{aRABEBK} \\ 
\cline{2-5} 
& IT & CPU & IT & CPU \\ 
\hline  
$1000 \times 500$   &  4879  & 0.976 &  3468  & 0.506  \\  
\hline  
$500 \times 1000$   &  4564  & 0.947 &  3202  & 0.438  \\ 
\hline  
$2000 \times 1000$  & 8639  &6.807 & 6268  &4.788  \\  
\hline  
$1000 \times 2000$  & 8370  &6.382 & 6759  &4.600  \\  
\hline  
$4000 \times 2000$  & 15002  &65.107 &  12983  &54.854  \\   
\hline  
$2000 \times 4000$  & 15917  &67.942 & 13176  &56.959 \\   
\hline  
\end{tabular}  
% }   
\end{table}
From Table~\ref{table:mini_Gaussian}, it is evident that aRABEBK consistently outperforms REABK in terms of both iteration count and computational time across all tested matrix sizes. For instance, when $m \times n = 4000 \times 2000$, aRABEBK completes the task in only $12{,}983$ iterations and $54.854$ seconds, whereas REABK requires $15{,}002$ iterations and $65.107$ seconds.

Figure~\ref{fig:mini_Guassian} presents the convergence curves of the relative error versus the number of iterations for the REABK and aRABEBK methods on Gaussian matrices of six different sizes. The results clearly show that aRABEBK (black curves) converges significantly faster than REABK in all tested cases.

Note that, in the minimum-norm least-squares setting, the cRABEBK method coincides with REABK. Therefore, the experimental results reported in Table~\ref{table:mini_Gaussian} and Figure~\ref{fig:mini_Guassian} convincingly demonstrate the effectiveness of the proposed adaptive
relaxation strategy, with consistent improvements over both REABK and its constant-relaxation variant.

\begin{figure}[!htbp] 
\centering
	\subfigure[$m=1000,n=500$]   
	{
		\begin{minipage}[t]{0.31\linewidth}
			\centering
			\includegraphics[width=1\textwidth]{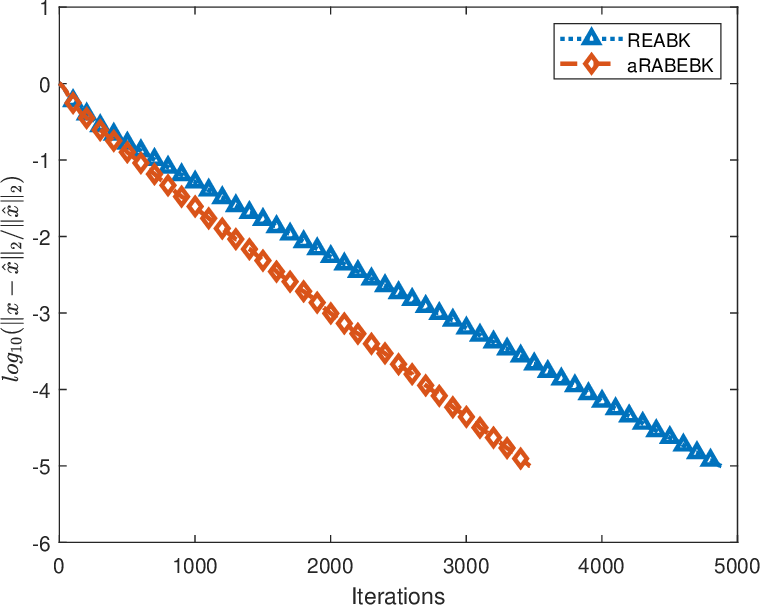}
		\end{minipage}
	}
	\subfigure[$m=2000,n=1000$]  
	{
		\begin{minipage}[t]{0.31\linewidth}
			\centering
			\includegraphics[width=1\textwidth]{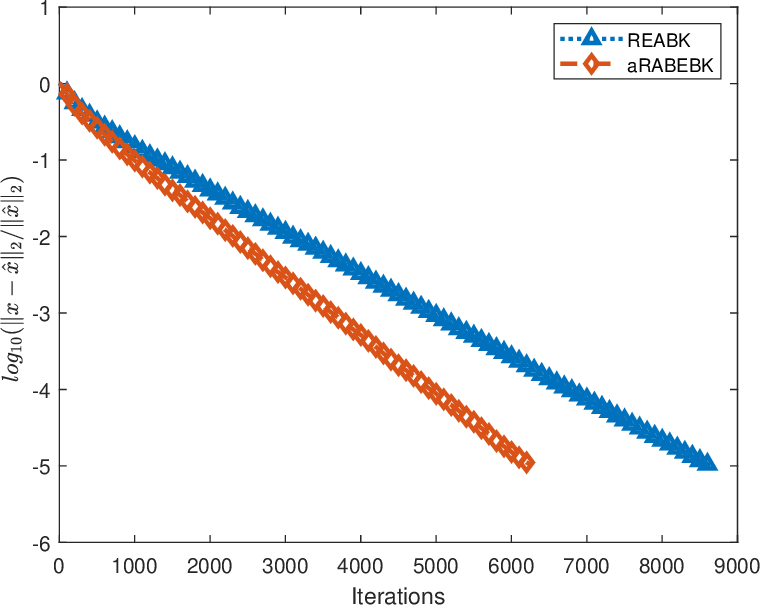}
		\end{minipage}
      }
      \subfigure[$m=4000,n=2000$]   
	{
		\begin{minipage}[t]{0.31\linewidth}
			\centering
			\includegraphics[width=1\textwidth]{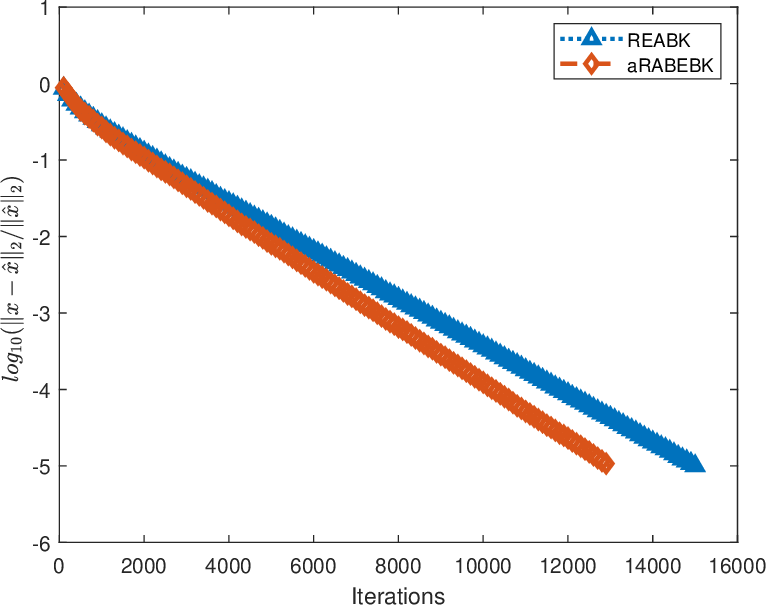}
		\end{minipage}
	}
  
	\subfigure[$m=500,n=1000$]   
	{
		\begin{minipage}[t]{0.31\linewidth}
			\centering
			\includegraphics[width=1\textwidth]{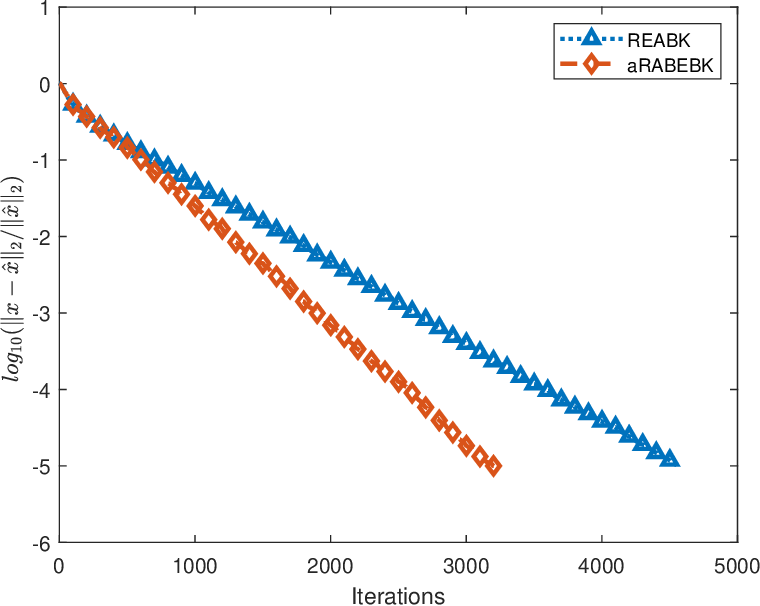}
		\end{minipage}
	}
	\subfigure[$m=1000,n=2000$]  
	{
		\begin{minipage}[t]{0.31\linewidth}
			\centering
			\includegraphics[width=1\textwidth]{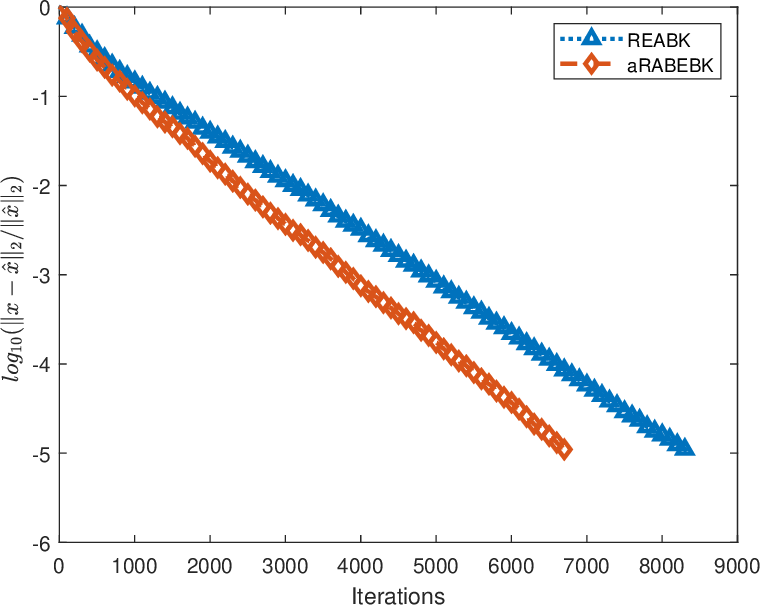}
		\end{minipage}
      }
      \subfigure[$m=2000,n=4000$]   
	{
		\begin{minipage}[t]{0.31\linewidth}
			\centering
			\includegraphics[width=1\textwidth]{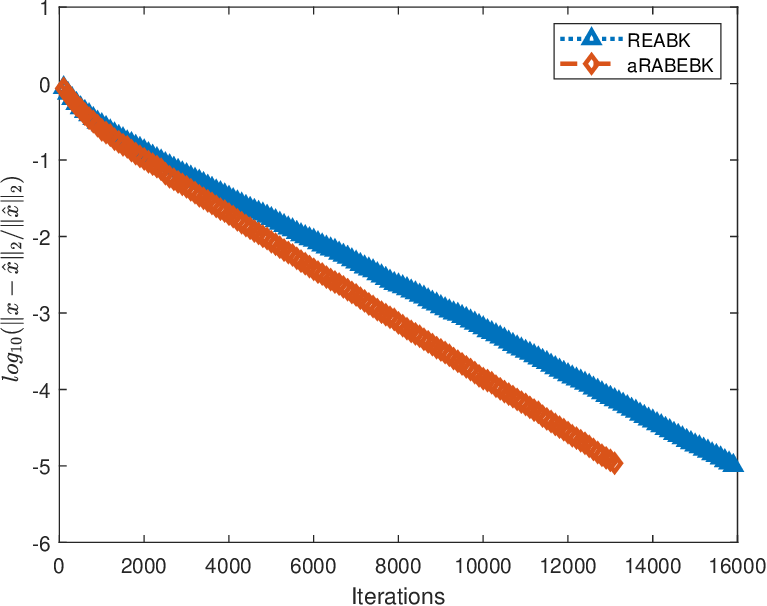}
		\end{minipage}
	}
  
\caption{The curves of relative error versus the number of iterations for overdetermined Gaussian matrices (top) and underdetermined Gaussian matrices (bottom) in the minimum-norm least-squares case.} 	
		\label{fig:mini_Guassian}
\end{figure}

\subsubsection{Experiments on structured random matrix} \label{subsubsub:rank-deficient}
Next, we generate the coefficient matrices using structured random matrices, following the procedure described in \cite{du2020randomized}, which we briefly summarize here for completeness.

For given $m$, $n$, the desired rank $r = \textrm{rank}( A)$, and a condition number $\kappa>1$, we construct a matrix $ A$ by $$ A = UDV^\top,$$ where $  U\in r^{m\times r} $ and $V\in \mathbb{R}^{n\times r}$. Entries of $ U$ and $ V$ are generated from a standard normal distribution, and then, columns are orthonormalized, $${\tt [U,\sim] = qr(randn(m,r),0);\qquad  [V,\sim] = qr(randn(n,r),0);}$$ The matrix $  D$ is an $r\times r$ diagonal matrix whose diagonal entries are uniformly distributed numbers in $(1,\kappa)$, $${\tt D = diag(1+(\kappa-1).*rand(r,1));}$$ with this construction, the condition number of \( A \) is upper bounded by \( \kappa \).

Table \ref{table:rank_deficient} reports the iteration counts and CPU times  for REBK, cRABEBK and aRABEBK in the sparse least-squares setting. Both RABEBK variants substantially outperform REBK, demonstrating the benefit of projecting onto multiple selected blocks simultaneously. Moreover, the adaptive method aRABEBK consistently achieves the best performance. For instance, in the fourth test case with $m\times n = 1000\times 2000$, aRABEBK reaches the desired accuracy in only 8.4 seconds, compared to 15.8 seconds for cRABEBK and 113.4 seconds for REBK.

\begin{table}[!htbp] 
\scriptsize 
\centering 
\caption{Numerical results for constructed matrices in the sparse least-squares setting.}
\label{table:rank_deficient}  
% \resizebox{\textwidth}{!}{  
\begin{tabular}{|c|c|c|c|c|c|c|c|c|c|c|c|}
\hline
\multirow{2}{*} {$m\times n$} & \multirow{2}{*} {rank} &\multirow{2}{*} {$\kappa$} 
&\multicolumn{2}{c}{REBK} &\multicolumn{2}{|c}{cRABEBK} &\multicolumn{2}{|c|}{aRABEBK} \\ 
\cline{4-9} 
& &  & IT &CPU  &IT &CPU &IT &CPU \\ 
\hline  
$1000\times 500$  &	480 &	10 &	 94413&	 6.219&	  7249&	2.112&	 5051&	1.164  \\  
\hline  
$500\times 1000$  &	480 &	10 &	 210660&	 12.205&	 	 16126&	 5.385&	 11043&	2.777  \\ 
\hline  
$2000\times 1000$ &	900 &	5 &	 83623&	57.312 &	 	 5637&	7.132&	 4263&	4.749  \\  
\hline  
$1000\times 2000$  &	900 &	5 &	 130349&	 113.437&	 	 8655&	15.761&	 6757&	8.435  \\   
\hline  
$4000\times 2000$  &	1500 &	2 &	 313228&	 1322.141&	 	 19800&	99.936&	 14957&	73.886  \\   
\hline  
$2000\times 4000$  &	1500 &	2 &	 764331&	 3008.484&	 	 47675&	264.677&	 39278&	206.083  \\   
\hline  
\end{tabular}  
% }   
\end{table} 

\begin{figure}[!htbp] 
\centering
	\subfigure[$m=1000,n=500,r=480,\kappa=10$]   
	{
		\begin{minipage}[t]{0.31\linewidth}
			\centering
			\includegraphics[width=1\textwidth]{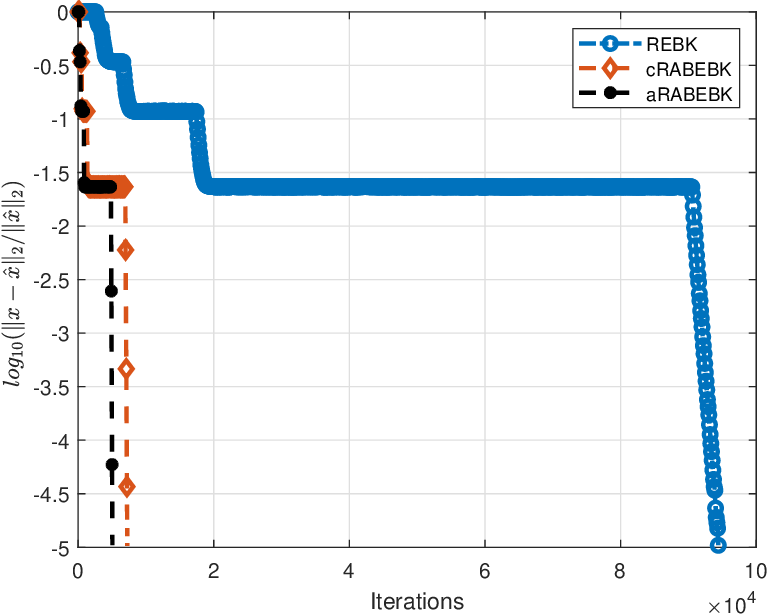}
		\end{minipage}
	}
	\subfigure[$m=2000,n=1000,r=900,\kappa=5$]  
	{
		\begin{minipage}[t]{0.31\linewidth}
			\centering
			\includegraphics[width=1\textwidth]{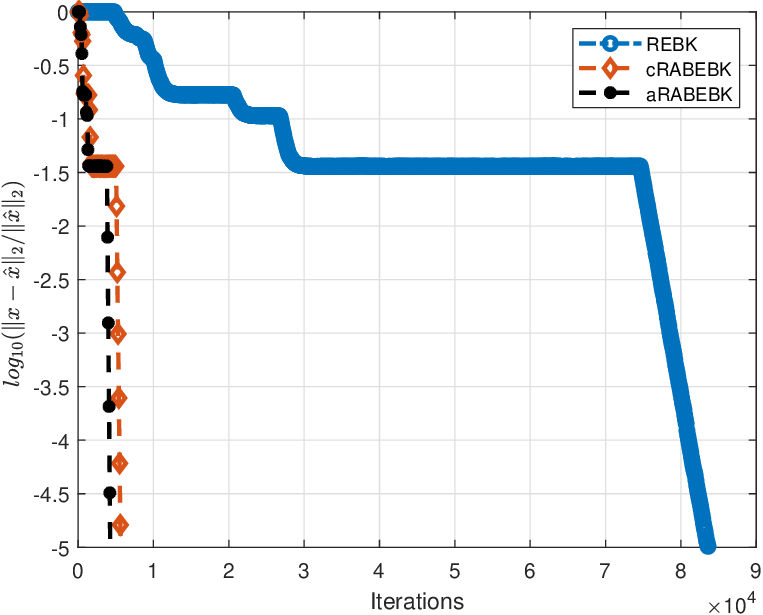}
		\end{minipage}
      }
      \subfigure[$m=4000,n=2000,r=1500,\kappa=2$]   
	{
		\begin{minipage}[t]{0.31\linewidth}
			\centering
			\includegraphics[width=1\textwidth]{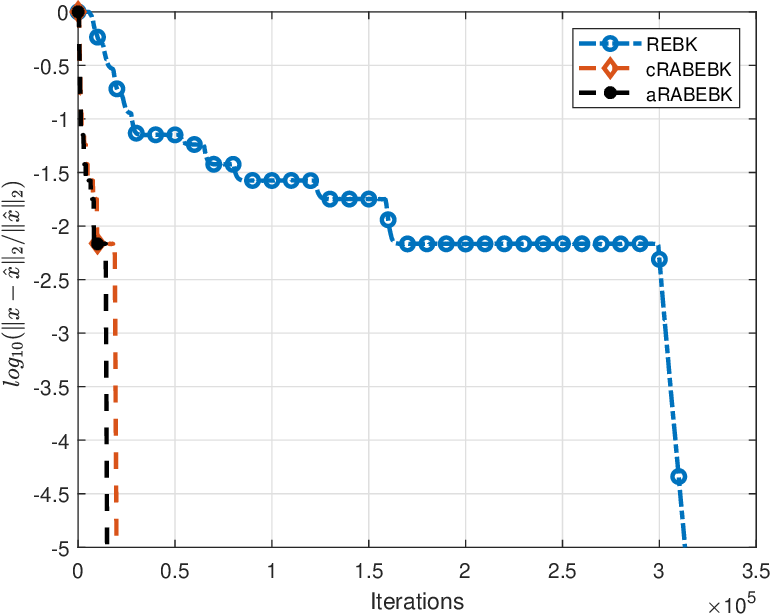}
		\end{minipage}
	}
  
	\subfigure[$m=500,n=1000,r=480,\kappa=10$]   
	{
		\begin{minipage}[t]{0.31\linewidth}
			\centering
			\includegraphics[width=1\textwidth]{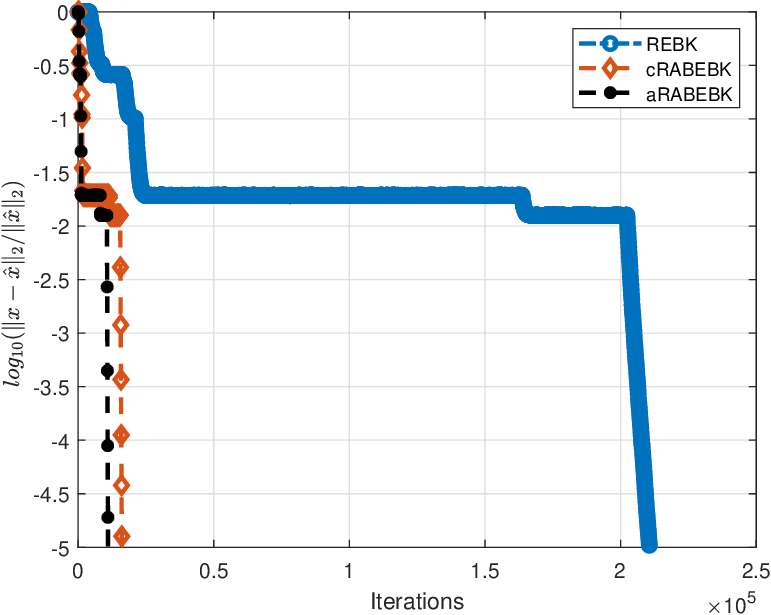}
		\end{minipage}
	}
	\subfigure[$m=1000,n=2000,r=900,\kappa=5$]  
	{
		\begin{minipage}[t]{0.31\linewidth}
			\centering
			\includegraphics[width=1\textwidth]{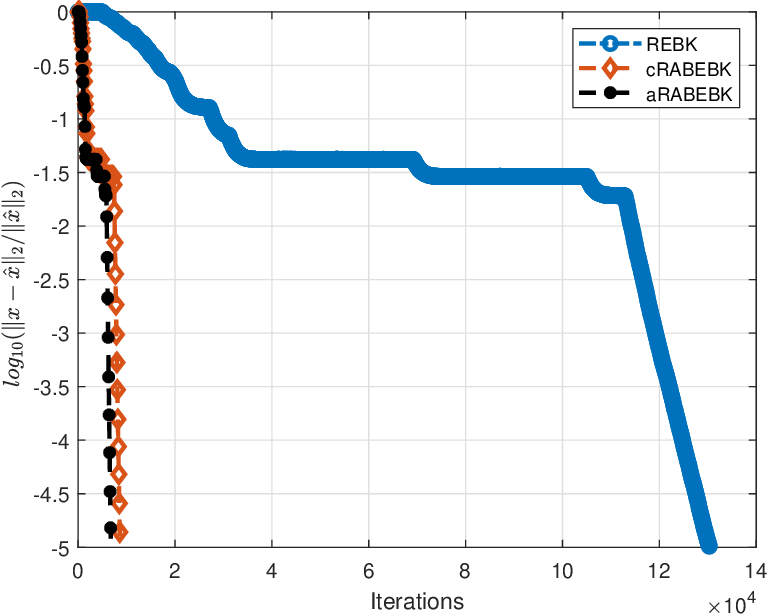}
		\end{minipage}
      }
      \subfigure[$m=2000,n=4000,r=1500,\kappa=2$]   
	{
		\begin{minipage}[t]{0.31\linewidth}
			\centering
			\includegraphics[width=1\textwidth]{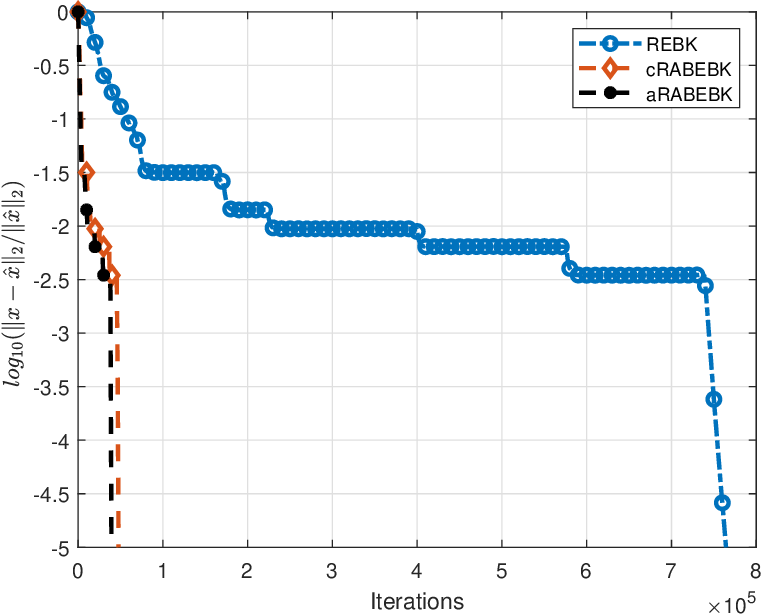}
		\end{minipage}
	}
  
\caption{The curves of relative error versus the number of iterations for overdetermined constructed matrices (top) and underdetermined constructed matrices (bottom) in the sparse least-squares setting.} 	
		\label{fig:rank_deficient}
\end{figure}

Figure~\ref{fig:rank_deficient} depicts the convergence curves of the relative error versus the number of iterations for the REBK, cRABEBK, and aRABEBK methods, illustrating their behaviour on structured random matrices of six different sizes in the sparse least-squares setting. It is evident from the figure that the adaptive method aRABEBK (black curves) converges significantly faster than its counterparts across all tested sizes, highlighting the combined benefit of the block-averaging
strategy and adaptive relaxation in accelerating convergence.

We next turn to the minimum-norm least-squares problem. Table~\ref{table:mini_rank_deficient} reports the numerical performance of the REABK and aRABEBK methods for structured rank-deficient matrices with various sizes, ranks, and condition numbers~$\kappa$.
The reported results include the iteration counts and CPU time required to reach the prescribed accuracy. We reiterate that, in the minimum-norm least-squares setting, the constant-relaxation variant cRABEBK coincides with REABK and is therefore not reported separately.

As shown in Table~\ref{table:mini_rank_deficient}, aRABEBK consistently outperforms REABK in terms of both iteration count and computational time across all test cases. For example, when $m \times n = 1000 \times 500$ with $\kappa = 10$, aRABEBK converges in $10{,}327$ iterations and $1.593$ seconds, whereas REABK requires $15{,}580$ iterations and $3.271$ seconds, corresponding to approximately a twofold reduction in runtime. These results further confirm that the adaptive relaxation strategy provides a clear advantage over REABK, as well as its constant-relaxation counterpart cRABEBK, in the minimum-norm least-squares setting.

\begin{table}[!htbp] 
% \scriptsize 
\centering 
\caption{Numerical results for constructed matrices in minimum-norm least-squares case.}\label{table:mini_rank_deficient}  
% \resizebox{\textwidth}{!}{  
\begin{tabular}{|c|c|c|c|c|c|c|}
\hline
\multirow{2}{*}{$m \times n$} & \multirow{2}{*}{rank} & \multirow{2}{*}{$\kappa$} 
& \multicolumn{2}{c|}{REABK} & \multicolumn{2}{c|}{aRABEBK} \\
\cline{4-7}
& & & IT & CPU & IT & CPU \\
\hline  
$1000 \times 500$   & 480  & 10 &  15580  & 3.271 &  10327  & 1.593 \\
\hline  
$500 \times 1000$   & 480  & 10 & 15020  & 2.995 & 9603  & 1.231 \\
\hline  
$2000 \times 1000$  & 900  &  5 &  7435  & 6.370 &  5626  & 4.343 \\
\hline  
$1000 \times 2000$  & 900  &  5 &  7504  &5.955 &  5659  & 3.869 \\
\hline  
$4000 \times 2000$  & 1500 &  2 & 2893  &11.583 & 2499  &9.659 \\
\hline  
$2000 \times 4000$  & 1500 &  2 & 2820  &11.271 & 2540 &9.562 \\
\hline  
\end{tabular}  
% }   
\end{table}

Figure~\ref{fig:mininorm_constructed} illustrates the convergence behaviour of the REABK and aRABEBK methods for minimum-norm least-squares problems. In both overdetermined and underdetermined settings, aRABEBK consistently exhibits faster convergence, highlighting the effectiveness of the proposed adaptive step size selection.

\begin{figure}[!htbp] 
\centering
	\subfigure[$m=1000,n=500$]   
	{
		\begin{minipage}[t]{0.31\linewidth}
			\centering
			\includegraphics[width=1\textwidth]{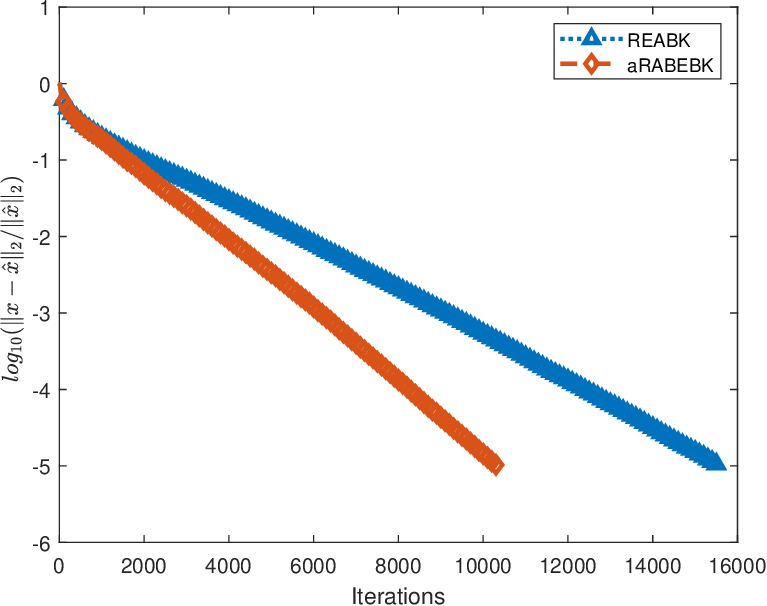}
		\end{minipage}
	}
	\subfigure[$m=2000,n=1000$]  
	{
		\begin{minipage}[t]{0.31\linewidth}
			\centering
			\includegraphics[width=1\textwidth]{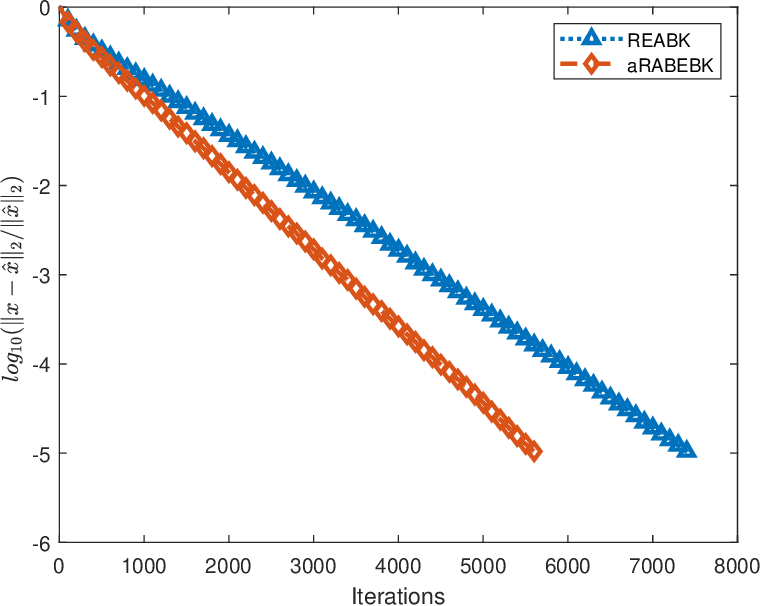}
		\end{minipage}
      }
      \subfigure[$m=4000,n=2000$]   
	{
		\begin{minipage}[t]{0.31\linewidth}
			\centering
			\includegraphics[width=1\textwidth]{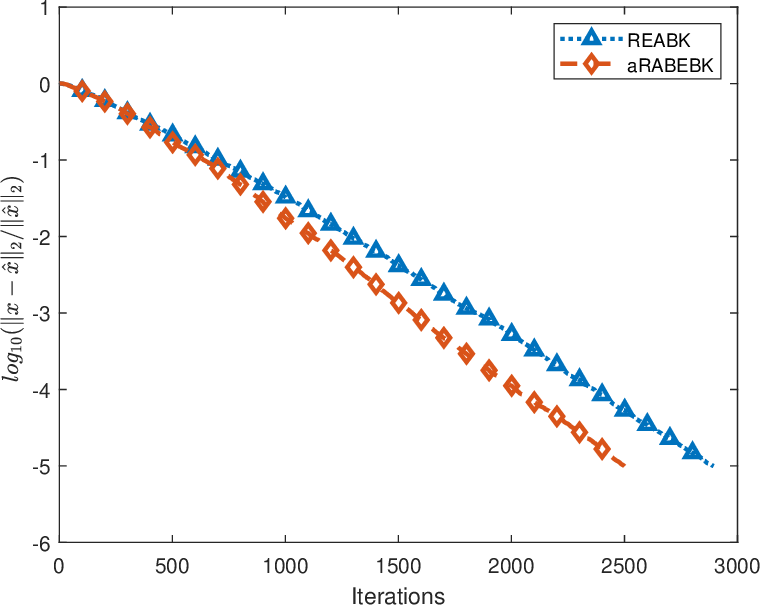}
		\end{minipage}
	}
  
	\subfigure[$m=500,n=1000$]   
	{
		\begin{minipage}[t]{0.31\linewidth}
			\centering
			\includegraphics[width=1\textwidth]{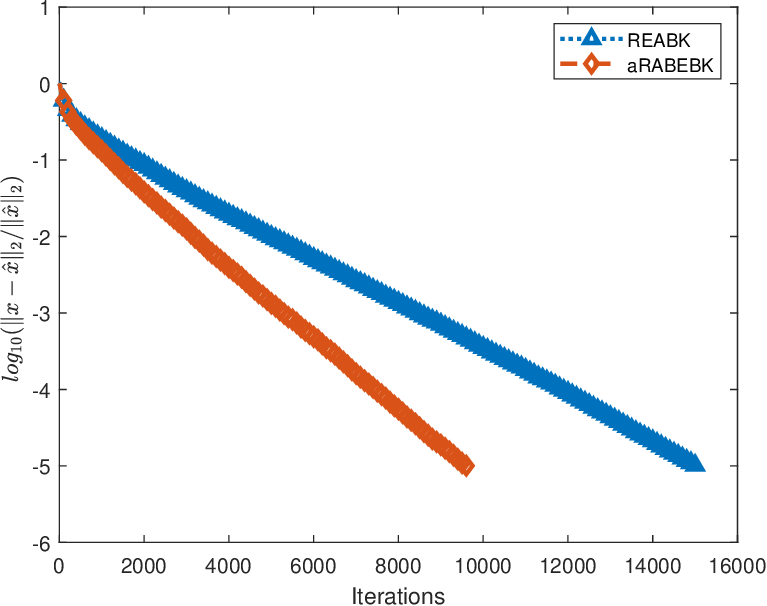}
		\end{minipage}
	}
	\subfigure[$m=1000,n=2000$]  
	{
		\begin{minipage}[t]{0.31\linewidth}
			\centering
			\includegraphics[width=1\textwidth]{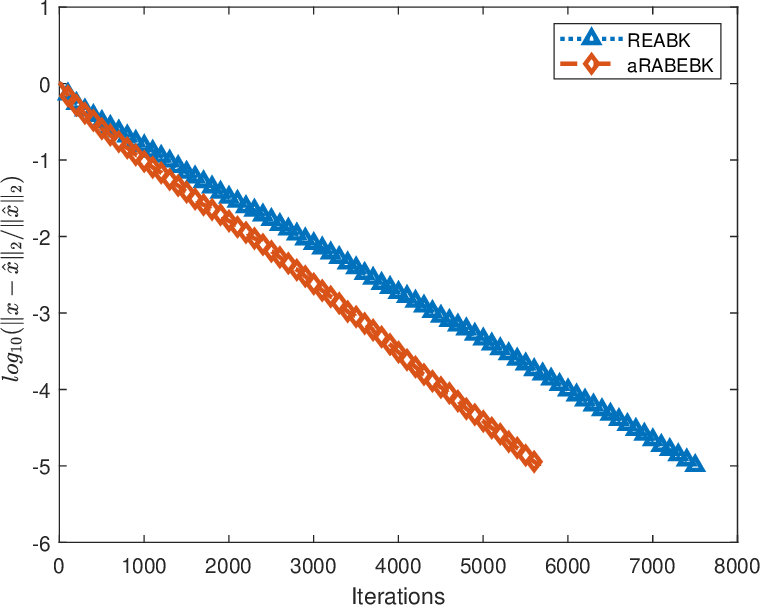}
		\end{minipage}
      }
      \subfigure[$m=2000,n=4000$]   
	{
		\begin{minipage}[t]{0.31\linewidth}
			\centering
			\includegraphics[width=1\textwidth]{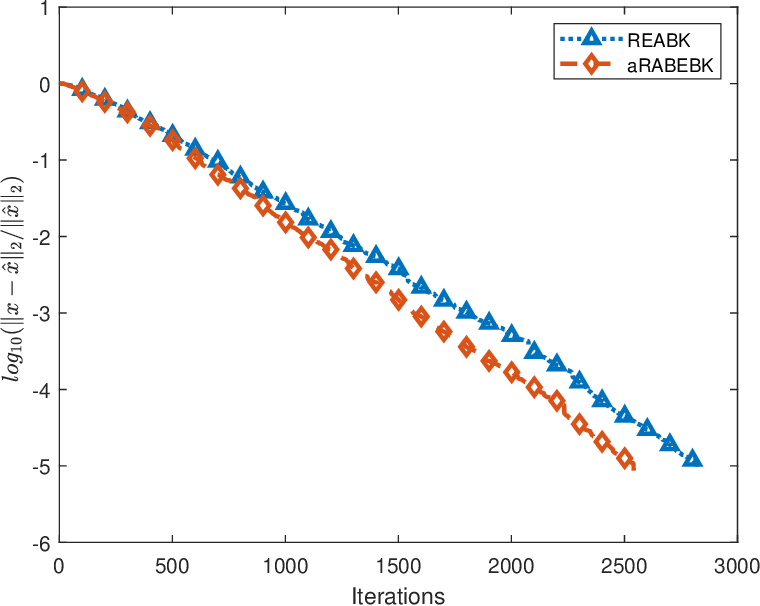}
		\end{minipage}
	}

\caption{The curves of relative error versus the number of iterations for overdetermined constructed matrices (top) and underdetermined constructed matrices (bottom) in the minimum-norm least-squares setting.} 	
		\label{fig:mininorm_constructed}
\end{figure}

The numerical results on structured rank-deficient problems consistently demonstrate the superior efficiency of the proposed adaptive relaxation strategy across both sparse and minimum-norm least-squares settings.

\subsection{Experiment on image recovery}
To further demonstrate the effectiveness of the proposed methods in practical scenarios, we consider an image recovery task based on the MNIST dataset, which is widely used in machine learning and computer vision. The image recovery of the MNIST dataset is available in the Tensorflow framework \cite{abadi2016tensorflow}. 

We consider two types of image recovery problems. In the sparse least-squares setting, the coefficient matrix \( A \in \mathbb{R}^{500 \times 784} \) is generated with i.i.d.\ Gaussian entries, yielding an underdetermined system that admits infinitely many solutions, among which sparsity-promoting behavior is of primary interest. In the minimum-norm least-squares setting, we set \( A \in \mathbb{R}^{2000 \times 784} \), resulting in an overdetermined system whose solution is characterized by the minimum Euclidean norm.

In both settings, the right-hand side vector is constructed as $$b = A \hat{x} + e,$$ where \( \hat{x} \in \mathbb{R}^{784} \) is obtained by vectorizing an MNIST image of size \(28 \times 28\), and \( e \) is an additive noise vector generated following the same noise construction procedure used throughout the numerical experiments.

The peak signal-to-noise ratio (PSNR) is used to measure the quality of the recovered image. It is defined as
\begin{equation}
\text{PSNR} := 10 \log_{10} \left( \frac{ \sum_{i=1}^n x_i^2 }{ \sum_{i=1}^n (x_i - \hat{x}_i)^2 } \right),
\end{equation}
where \(  \hat{x} \in \mathbb{R}^n \) denotes the ground-truth image and \( x \in \mathbb{R}^n \) is the reconstructed image.

%This definition is equivalent to the classical PSNR up to a constant scaling factor when the image intensity range is fixed.

\begin{figure}[!htbp]
  \centering
  \includegraphics[width=1\linewidth]{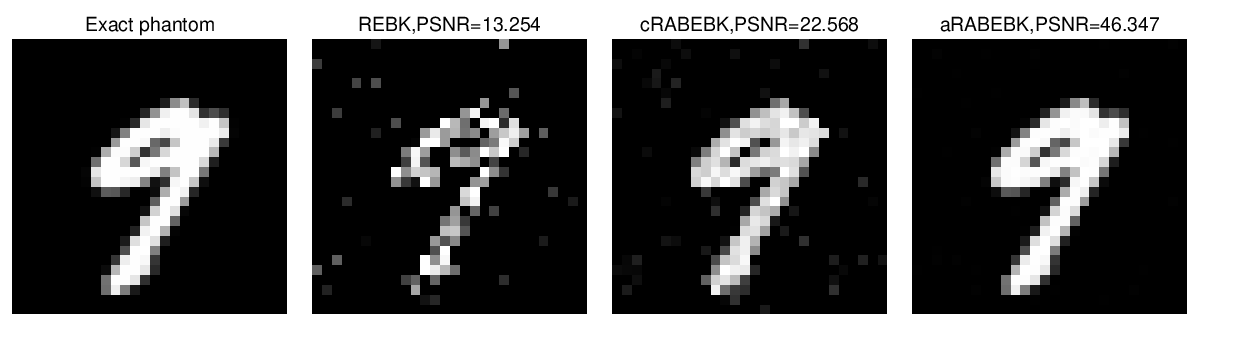}
\caption{Original MNIST image and reconstructions obtained by the REBK, cRABEBK, and aRABEBK methods in the sparse least-squares setting.}
 \label{fig:mnist_our_noise}
\end{figure}
Figure \ref{fig:mnist_our_noise} presents the reconstruction results of three test methods after running 10000 iterations in the sparse least-squares setting. Our aRABEBK methods produces significantly higher-quality reconstructions, achieving a PSNR of 46.35 compared to 13.25 and 22.59 for REBK and cRABEBK, respectively. 

\begin{figure}[!htbp]
  \centering
  \includegraphics[width=1\linewidth]{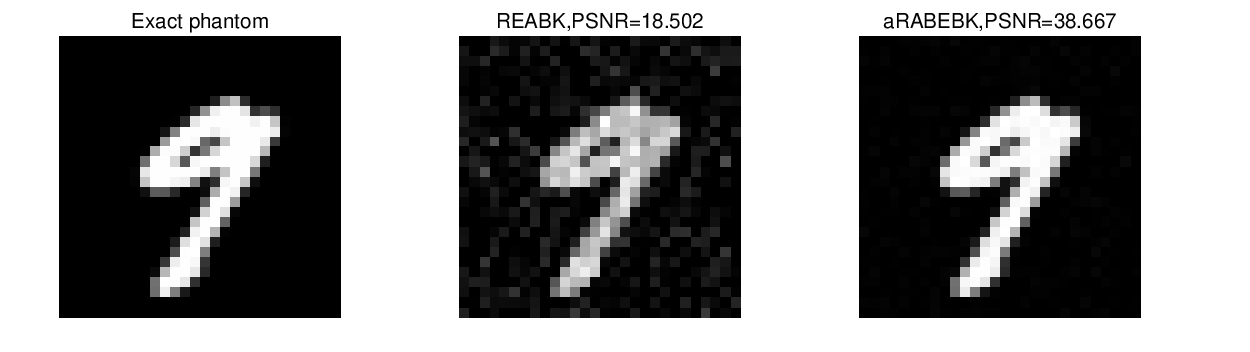}
\caption{Original MNIST image and reconstructions obtained by the REABK and aRABEBK methods in the minimum-norm least-squares setting.}
 \label{fig:mnist_mininorm}
\end{figure}
Figure~\ref{fig:mnist_mininorm} presents the reconstructed MNIST images obtained by the tested methods after 1000 iterations in the minimum-norm least-squares setting. It can be observed that aRABEBK produces a much clearer reconstruction with significantly fewer visual artifacts than its counterparts. In particular, aRABEBK achieves a PSNR of 38.67, whereas REABK attains only 18.50.

Both the visual quality and the quantitative PSNR results consistently demonstrate that the adaptive relaxation strategy substantially improves reconstruction accuracy. These findings further confirm the robustness and effectiveness of aRABEBK in handling real-world inverse problems with noisy data arising from sparse and minimum-norm least-squares formulations.

\section{Conclusion} \label{sec:conclusion}
We introduce the \emph{adaptive randomized averaging block extended Bregman-Kaczmarz} (aRABEBK) method for solving general composite optimization problems. Adaptive relaxation parameters dynamically adjust the contribution of each selected block, fully exploiting block-averaged information and enabling accelerated convergence.
Numerical experiments demonstrate that aRABEBK outperforms existing Kaczmarz-type methods in both iteration count and computational efficiency. 

The proposed framework is highly flexible and can be applied to a broad class of inverse problems beyond standard least-squares formulations, including structured recovery tasks such as sparse phase retrieval and other imaging applications. Future work including extending the adaptive strategy and investigating its performance on large-scale, real-world inverse problems is deserved to further study.\\

\noindent \textbf{Funding} The work of the authors was partially supported by National Natural Science Foundation of China grants 12471357 and the China Postdoctoral Science Foundation under Grant Number 2025M783127.\\

\noindent \textbf{Data availability} No new data were created or analyzed in this study.

\section*{Declarations}
\noindent\textbf{Competing Interests} The authors declare that they have no conflict of interest.

\bibliographystyle{plain}
\bibliography{ref_adpRABEBK}
\end{document}